\documentclass[12pt]{amsart}
\usepackage{amssymb,amsmath,amsthm}
\usepackage{verbatim}
\usepackage{pinlabel}
\usepackage{graphicx}
\usepackage{enumerate}
\usepackage{dsfont}
\usepackage{simplemargins}
\usepackage{hyperref}

\hypersetup{
  colorlinks   = true,    
  urlcolor     = black,    
  linkcolor    = black,    
  citecolor    = black      
}

\settopmargin{1in}
\setbottommargin{1in}

\renewcommand{\epsilon}{\varepsilon}
\renewcommand{\setminus}{\smallsetminus}
\let\inf\relax \DeclareMathOperator*\inf{\vphantom{p}inf}

\theoremstyle{plain}
\newtheorem{theorem}{Theorem}
\newtheorem*{theorem*}{Theorem}
\newtheorem{corollary}{Corollary}[section]
\newtheorem{lemma}[corollary]{Lemma}
\newtheorem{proposition}[corollary]{Proposition}

\theoremstyle{definition}
\newtheorem*{definition}{Definition}

\theoremstyle{remark}
\newtheorem*{remark}{Remark}

\newcommand{\Z}{{\mathds Z}}
\newcommand{\N}{{\mathds N}} 
\newcommand{\E}{{\mathds E}}
\newcommand{\R}{{\mathds R}}
\newcommand{\C}{{\mathds C}}

\newcommand{\Disk}{{\mathds D}}
\newcommand{\D}{\Disk}

\newcommand{\bd}{\partial}

\newcommand{\diam}{{\rm diam}}
\newcommand{\area}{{\rm area}}
\renewcommand{\Re}{{\rm Re}}
\renewcommand{\Im}{{\rm Im}}

\newcommand{\ceq}{\;{\mathrel{\mathop:}=}\;}

\newcounter{todos}

\setallmargins{1in}

\begin{document}

\title[Discrete analytic functions on lattices without global geometric control]{Discrete analytic functions on non-uniform lattices without global geometric control}
\author{Brent M. Werness}
\begin{abstract}
Recent advances in the study of conformally invariant discrete random processes have lead to increasing interest in the study of discrete analogues to holomorphic functions.  Of particular interest are results which provide conditions under which these discrete functions can be shown to converge to continuum versions as the lattice spacing shrinks to zero.  Recent work by Skopenkov \cite{Skopenkov} has extended these results to include a wide class of non-uniform quadrilateral lattices with a pair of regularity conditions, one local and one global.  

Such a result is sufficient for the study of random processes on deterministic lattices, however to establish convergence results for conformally invariant random processes on random triangulations, such a global regularity condition cannot be assumed.  In this paper we provide a convergence result on quadrilateral lattices upon which we enforce only a local condition on the geometry of each face.
\thanks{This project was partially funded by NSF grant DMS-1304163.}
\end{abstract}
\maketitle

\section{Introduction}
Discrete complex analysis has a long history dating back to the work of \cite{Isaacs,Ferrand,Duffin}, and continuing to the present day through the work of \cite{Mercat1,ChelSmir,SmirDisc}.  In recent years, there has been added interest in the field from the point of view of conformally invariant discrete random processes, where discrete complex analysis has proven quite useful in many attempts to prove convergence results of these processes towards their continuum analogues (Schramm--Loewner Evolutions) or in greatly sharpening the understanding of discrete observables of that model \cite{SmirDisc,conflatt,connective,ising,perc,kenyon1,kenyon2,kenyon3,LERW,IsingSpin}.

From this point of view, the most useful form of discrete complex analysis that has emerged has been the approach focused on rhombic lattices or on their duals, isoradial graphs, which are a collection of convex polygons in which every face can be inscribed in a circle \cite{Kenyon4,SmirDisc,ChelSmir,ising}.  In this setting, there is a robust set of theorems which ensure that limits of discrete analytic functions are analytic \cite{Cour,ChelSmir}.

Recent work by Skopenkov has extended these results by providing a broad set of conditions under which sequences of discrete harmonic or holomorphic functions on a sequence of lattices converges to a harmonic function in the continuous domain \cite{Skopenkov}.  Rather than work with rhombic lattices, Skopenkov works in the greater generality of orthogonal lattices, which are collections of quadrilaterals where the diagonals of each face are orthogonal to each other.  We will discuss the specific conditions for convergence of such functions in greater detail in Section~\ref{ExactSec}, however, roughly speaking, Skopenkov proved that if you have a sequence of lattices which converge to the continuous domain and there is uniform control on the geometry of each face \emph{and} uniform global control on the density of vertices of your quadrangulation, then you may conclude convergence of the appropriate discrete harmonic functions to their continuum counterparts. 

In this paper, we provide an alternative set of conditions which ensure convergence.  Following Skopenkov, we will still work in the setting of orthogonal lattices.  However, we will differ in our conditions by requiring only a single \emph{local} control on the regularity of the quadrilateral lattice, without any need for uniform global control on the density of vertices.  Slightly informally stated, our main result is as follows.
  
\begin{theorem*}
	Let $\Omega \subseteq \C$ be a bounded simply-connected domain.  Let $g:\C \rightarrow \R$ be smooth boundary values.  Let $Q_n$ be a sequence quadrilateral lattices with orthogonal diagonals and uniformly bounded angles and ratios amongst the edge an diagonal lengths converging to $\Omega$.  Then the solutions to the discrete Dirichlet problem with boundary values $g$ converge to the solution to the Dirichlet problem in $\Omega$ with boundary values $g$.
\end{theorem*}

Finding generalizations of these convergence results to more general lattices has been a question of interest, and was stated as an open question by Smirnov \cite{SmirDisc}.  However, aside from this general interest in attempting to find convergence results in the most general possible setting, there is a specific probabilistic motivation for examining this particular direction generalization.  In recent years, there has been significant progress in understanding the structure of large random planar maps, from the metric space view \cite{BuziosNotes,LeGallRev}, from the point of view of the conjectured conformal structure, and through the couplings of conformally invariant statistical physics models with these surfaces \cite{HCSheff,CWSheff,QLE,KPZ,TreeMate1,TreeMate2,TreeMate3}.  The most direct route towards proving rigorous results on the convergence of discrete models on random planar maps to Schramm--Loewner Evolutions requires results on the convergence of discrete analytic functions flexible enough to handle the highly irregular lattices that occur when considering natural embeddings of planar maps.  These results are in this direction, and indeed we will prove in Section~\ref{MotSec} that there is a reasonable model of random planar maps to which our results may be applied, under the assumption that the diameter of the largest circle in the circle-packings of these surfaces tends to zero.

\subsection{Comparison to previous results}\label{ExactSec}

We will work entirely with \emph{quadrilateral lattices}, which is to say a planar graph $Q$ with a given straight-line embedding into $\C$ whose vertices are identified with a set $Q^* \subset \C$ such that every bounded face of $Q$ is a quadrilateral.  We will assume all faces are non-degenerate in the sense that the four edges are all distinct.  We freely allow non-convex quadrilaterals.  Throughout this paper, we will always assume that the quadrilaterals are \emph{orthogonal}, which is to say that the lines through the diagonals meet at a right angle.

Orthogonal lattices may be easily found both algorithmically and by hand.  To help emphasize this point, all lattices included in all figures will be exactly orthogonal.

\begin{figure}[ht!]
	\includegraphics{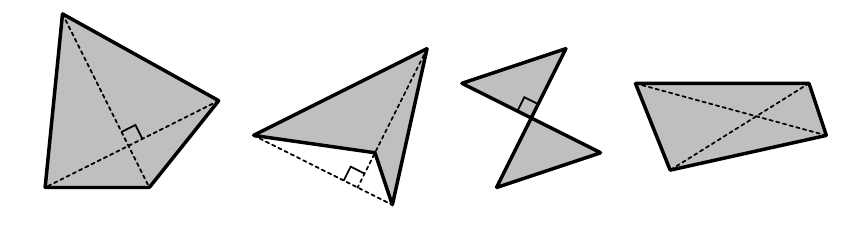}
	\caption{The first two quadrilaterals are admitted in our lattices as all edges are disjoint and the diagonals are orthogonal.  The third is not allowed since the edges are not disjoint (although the diagonals are orthogonal). The fourth is also not allowed since the diagonals are not orthogonal.}
	\label{AllowableFig}
\end{figure}

Since every bounded face of $Q$ is a quadrilateral, we know that the graph is bipartite and thus admits a two-coloring as black and white vertices.  Moreover, such a two-coloring is unique up to the interchange of black and white vertices (see Figure~\ref{TwoColorFig}).  Throughout this paper, the black vertices will be denoted by $Q^\bullet$, and the white vertices by $Q^\circ$.  Thus we have that the set of all vertices may be written $Q^* = Q^\bullet \sqcup Q^\circ$.  

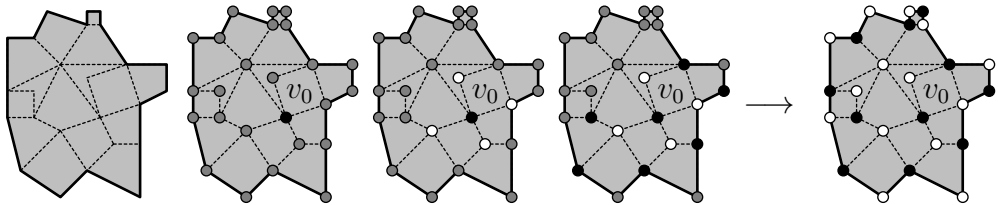
\begin{figure}[ht!]
	\labellist
	\pinlabel $v_0$ [l] at 115 45
	\pinlabel $v_0$ [l] at 185 45
	\pinlabel $v_0$ [l] at 255 45
	\pinlabel $v_0$ [l] at 355 45
	\pinlabel $\longrightarrow$ [c] at 300 40
	\endlabellist
	\includegraphics{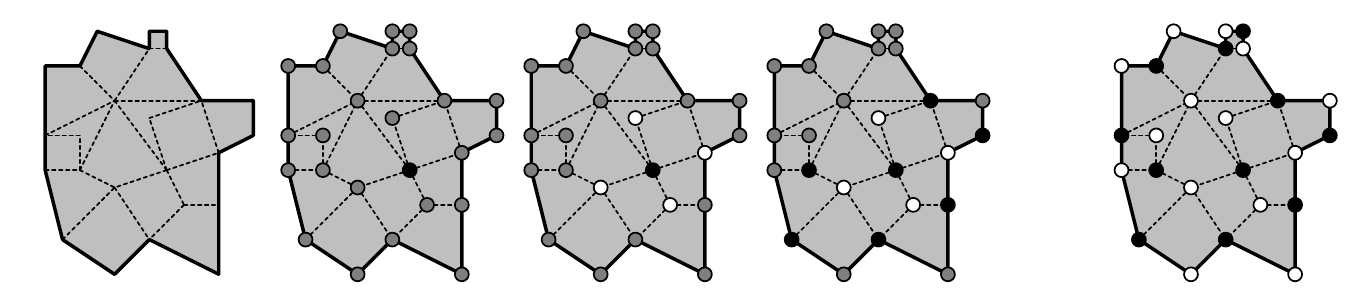}
	\caption{Any quadrilateral lattice can be uniquely two-colored (up to exchange of the two colors) since they are bipartite.  This can be seen by coloring an initial vertex $v_0$ black, coloring the neighbors of that vertex white, and the iterating the process for subsequent neighborhoods.  The faces being quadrilaterals provides a form of integrability condition ensuring that the lengths of all paths from $v_0$ to any target vertex are of the same parity, and thus the coloring is unique.}
	\label{TwoColorFig}
\end{figure}

We now provide the notion of convergence for a sequence of such quadrilateral lattices to a domain in $\C$.  We first review the Hausdorff distance between sets in the plane. 

\begin{definition}
	Given two set $X$ and $Y$ of a metric space, the \emph{Hausdorff distance} between them is the infimum over all $\epsilon > 0$ so that $X$ is contained in the $\epsilon$-neighborhood of $Y$ and $Y$ is contained in the $\epsilon$-neighborhood of $X$, or more formally written: 
	\[
	\mathrm{d}_{\mathrm{Haus}}(X,Y) = \max\bigl\{\sup_{x \in X} \inf_{y \in Y} d(x,y),\; \sup_{y \in Y} \inf_{x \in X} d(x,y)\bigr\}.
	\]
\end{definition}

Let $M(Q) \ceq \sup_{z \sim w} |z-w|$ denote the maximal edge length of a lattice $Q$.  We may now provide a precise definition of what is meant by approximation by quadrilateral lattices, which is a mild restatement of the definition given in \cite{Skopenkov}.

\begin{definition}\label{ApproxDef}
	Given a bounded simply-connected domain $\Omega$, we will say that a sequence of quadrilateral lattices $\{Q_n\}$ \emph{approximates} $\Omega$ if the following hold:
	\begin{itemize}
		\item $M(Q_n) \rightarrow 0$ as $n \rightarrow 0$, and
		\item $\mathrm{d}_{\mathrm{Haus}}(\bd Q_n, \bd \Omega) \rightarrow 0$ as $n \rightarrow 0$.
	\end{itemize} 
\end{definition}

This provides the conditions needed to say that a sequence of quadrilateral lattice approximates the shape of the domain.  This provides no control over the local geometry of the quadrilateral lattices.  If we want holomorphic or harmonic functions defined on these discrete structures to converge to their continuum analogues (defined and discussed rigorously in Section~\ref{DiscHoloSec}), it is natural to add additional conditions on the geometric regularity of these lattices. 

In \cite{Skopenkov}, the following conditions were assumed.

\begin{definition}\label{SkopLocalDef}
	We will say that a quadrilateral lattice is \emph{uniform non-degenerate} with constant $C>0$ if it satisfies the following two conditions:
	\begin{description}
		\item[Non-degenerate] The ratio of the lengths of the diagonals of every quadrilateral face is bounded above by $C$ (and thus below by $C^{-1}$) and the angle between the diagonals is bounded below by $C^{-1}$.
		\item[Uniform] The number of vertices within any ball of radius $M(Q)$ is less than $C$.
	\end{description}
\end{definition}

Under these regularity conditions on the approximating lattices, Skopenkov proved that the solutions to the discrete Dirichlet problem converge to those of the continuous problem.

\begin{theorem}[Skopenkov \cite{Skopenkov}]\label{SkopTheorem}
	Let $\Omega \subseteq \C$ be a bounded simply-connected domain.  Let $g:\C \rightarrow \R$ be a smooth function.  Let $Q_n$ be a sequence of orthogonal uniform non-degenerate lattices with constant $C$ approximating a domain $\Omega$.  Then the solution $u_{Q_n,g}:Q_n^* \rightarrow \R$ of the Dirichlet problem on $Q_n$ uniformly converges to the solution $u_{\Omega,g}$ of the Dirichlet problem on $\Omega$.
\end{theorem}

Let us examine the assumptions more closely.  The condition of non-degeneracy provides local control over the geometry of the quadrilateral lattice.  Such a condition seems likely necessary to assume if one wishes to prove a theorem of this form.  The condition of uniformity is however a very strong condition which excludes many reasonable lattices such as those in Figure~\ref{MyLatticesFig}.  We will provide a result allowing for removal of this condition as asked in \cite{Skopenkov}.  In particular, we will replace the two conditions of uniformity and non-degeneracy from \cite{Skopenkov} with the following single local condition.

\begin{figure}[ht!]
	\includegraphics[width=4in]{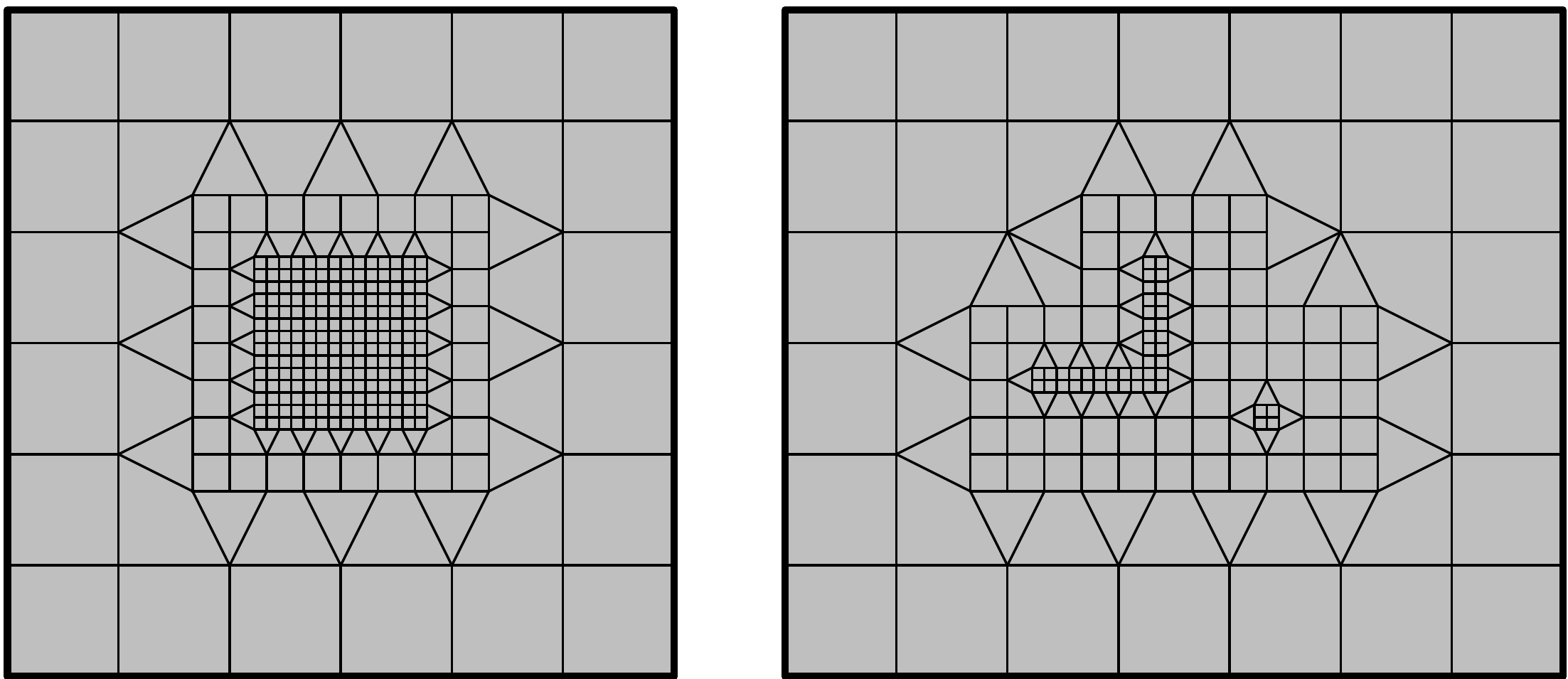}
	\caption{Two examples of orthogonal quadrilateral lattices of the type which satisfy the conditions of Theorem~\ref{MainTheorem}, but not Theorem~\ref{SkopTheorem}.  On the left, nested annuli of squares of differing sizes (each new level is made of squares one third the size of the previous generation) connected by an annuli of orthogonal quadrilaterals. This construction can be iterated any number of times and remain uniformly $K$-round ($K>2\pi/\cos^{-1}(4/5)\approx9.7\ldots$).  The right hand side shows a more complex set of nested domains constructed by the same procedure.  Any sequence of such quadrilateral lattices with the maximum edge length tending to zero will satisfy the conditions of Theorem~\ref{MainTheorem}, regardless of the size of the smallest quadrilateral.}
	\label{MyLatticesFig}
\end{figure}

\begin{definition}\label{RoundDef}
	A quadrilateral $f = [z_1z_2z_3z_4]$ is $K$-round if all interior angles are bounded below by $2\pi/K$ and the ratio of the lengths of any pair of edges is less than $K$.  We will say a lattice $Q$ (or a sequence of lattices $\{Q_n\}_{n \in \N}$) is $K$-round if all quadrilateral faces in the lattice (or sequence of lattices) are $K$-round for a fixed $K$.
\end{definition}

\begin{remark}
	Note that we know that $K \ge 4$ since otherwise there are not quadrilaterals with all interior angles greater than $2 \pi / K$.
\end{remark}

This is a single local condition, which contains the lattice shown above.  The main result of this paper is to show that this is still sufficient to prove the same result.

\begin{theorem}\label{MainTheorem} 
	Let $\Omega \subseteq \C$ be a bounded simply-connected domain.  Let $g:\C \rightarrow \R$ be a smooth function.  Let $Q_n$ be a sequence of $K$-round finite orthogonal lattices approximating a domain $\Omega$.  Then the solution $u_{Q_n,g}:Q_n^* \rightarrow \R$ of the Dirichlet problem on $Q_n$ uniformly converges to the solution $u_{\Omega,g}$ of the Dirichlet problem on $\Omega$.
\end{theorem}

It is worth exploring the relationship between this condition and the results of \cite{Skopenkov}.  It is clear that there are sequences of $K$-round quadrilateral lattices which do not satisfy the conditions uniformity and non-degeneracy with a uniform choice of constant $C$ as the control by condition $K$ is strictly local (see Figure~\ref{MyLatticesFig} for explicit example).  It is also the case that not every sequence of lattices which are uniform and non-degenerate with constant fixed constant $C$ are uniformly $K$-round.  Indeed Figure~\ref{NotKFig} illustrates that these conditions admit arbitrarily distorted quadrilaterals (in terms of $K$-roundness) of the $Q$ while still being uniformly uniform non-degenerate.  In this way, the theory presented here is a distinct generalization of the theory of \cite{ChelSmir} from the one provided by \cite{Skopenkov} which has been constructed to apply on lattices which contain some of the irregularities one would expect in the case of discrete quantum gravities.

\begin{figure}[ht!]
	\labellist
	\pinlabel $\epsilon$ [bc] at 48 46
	\endlabellist
	\includegraphics[width=1.2in]{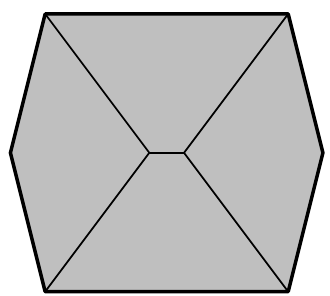}
	\caption{By sending epsilon to zero in comparison to the other edge lengths in the figure (while maintaining orthogonality of the quadrilaterals) one may produce quadrangulations of arbitrarily poor $K$-roundness, while remaining uniformly non-degenerate uniform when tiled to approximate a domain.}
	\label{NotKFig}
\end{figure}

\subsection{Motivations from random maps and circle packings}\label{MotSec}

These questions can be motivated from a number of angles.  For many, the motivation arise from questions in numerical analysis of holomorphic or harmonic functions.  The results of this paper are useful in this direction since we make no global assumptions and thus we may provide convergence results for \emph{adaptive meshes}, which are roughly lattices which are more finely subdivided in some region or interest, and only coarsely subdivided in regions where accurate estimates are not desired.  Theorem~\ref{MainTheorem} states that if you have a sequence of such adaptive meshes with some local geometric control and global control on the longest edge, then convergence to the true harmonic function is guaranteed.  Indeed, Figure~\ref{MyLatticesFig} shows a method to explicitly construct adaptive meshes with arbitrary regions of interest, although showing that such meshes provide a practical method of approximating harmonic functions would require significant additional analysis.

The primary motivation for this work does not come from numerical concerns, but rather from potential applications of this type of result to the study of random planar maps.  A \emph{planar map} is a planar graph together with an embedding in the sphere up to orientation preserving homeomorphisms of the sphere.  If the number of edges is fixed, then the number of maps is finite and one may consider the uniform measure on such a set.  One of the best studied models, and the one of most interest to us here are \emph{random quadrangulations} where all faces are restricted to be quadrilaterals.

The metric geometry of such large planar maps is fairly well understood \cite{BuziosNotes,LeGallRev}.  However, questions of the conformal geometry one obtains if one associates some form of conformal structure, say by considering the Riemann surface formed by filling each face of the quadrangulation with a unit square, is still only beginning to be understood (however recent work by Miller and Sheffield has produced significant progress in this area \cite{QLE,TreeMate1, TreeMate2, TreeMate3}).  

An alternative approach to attempting the understand the conformal geometry is to appeal to the theory of circle packing to provide a combinatorial theory of holomorphicity \cite{Koebe,Thurston,Andreev,RodinSullivan,CirclePack}.  Given a graph $G=(V,E)$, a \emph{circle packing} $\mathcal{P}$ of $G$ is a collection of disjoint open disks, one for each vertex $v \in V$, so that the closures of two disks associated with vertices $v$ and $w$ intersect if and only if $v \sim w$ in $G$.  Given a packing $\mathcal{P}$, one may associate a graph to it, called the \emph{nerve} of $\mathcal{P}$ by taking a vertex for each disk in $\mathcal{P}$, and an edge whenever the closures of the disks intersect.

One of the first and most important results in the theory is the well-known Koebe--Andreev--Thurston circle packing theorem.
\begin{theorem}[Circle Packing Theorem \cite{Koebe,Andreev,Thurston}]\label{CircTheorem}
	If $G$ is a simple planar graph, then there exists a packing $\mathcal{P}$ on the sphere whose nerve is $G$.  If the graph $G$ is a triangulation, then the packing is unique up to action by M\"obius transformations.
\end{theorem}

Such circle packings provide an alternative natural notion of discrete holomorphic maps.  For the sake of brevity, we provide only a brief non-rigorous discussion of this result. See \cite{RodinSullivan,CirclePack} for a detailed and rigorous proof.  Consider a domain $\Omega$ in $\C$ which for concreteness contains the points $0$ and $1$ and let $H_\epsilon$ denote sublattice the triangular lattice with lattice spacing $\epsilon > 0$ intersected with $\Omega$.  To avoid complications, we assume that $0$ and $1$ are  vertices of $H_\epsilon$ and that the domain $\Omega$ is chosen such that $H_\epsilon$ is is connected for all $\epsilon < \epsilon_0$, however the definitions may be extended so these conditions are not required.  We now extend this to a triangulation of the sphere by adjoining a point at infinity and connecting every vertex on the boundary of $H_\epsilon$ to that point.

\begin{figure}[ht!]
	\labellist
	\pinlabel $f_\epsilon$ [bc] at 200 110
	\pinlabel $0$ [l] at 52 108
	\pinlabel $1$ [l] at 92 108
	\pinlabel $f_\epsilon(0)$ [tc] at 310 -8
	\pinlabel $f_\epsilon(1)$ [tc] at 360 -8
	\endlabellist
	\includegraphics[width=5 in]{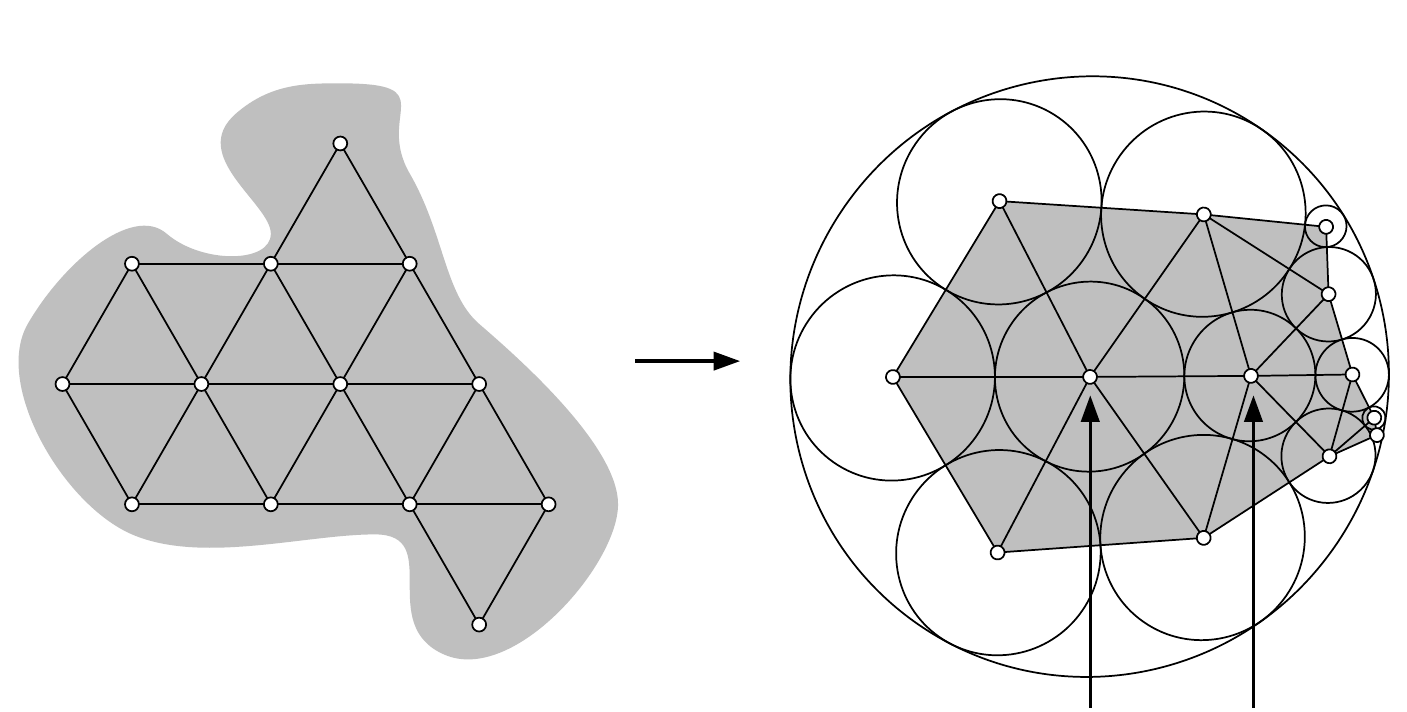}
	\vspace{0.2in}
	\caption{An illustration of using circle packings to approximate conformal maps via Theorem~\ref{RSThm}.  The grey domain on the left is approximated by a simply connected subset of the hexagonal lattice with side-length $\epsilon$.  Edges to the point at infinity have been suppressed for clarity.  This is then circle packed to produce the approximately conformal image on the right.  The limiting map as $\epsilon \rightarrow 0$ becomes the unique Riemann map.  The circle packing was produced using CirclePack \cite{CPSoft}.}
	\label{RSFig}
\end{figure}
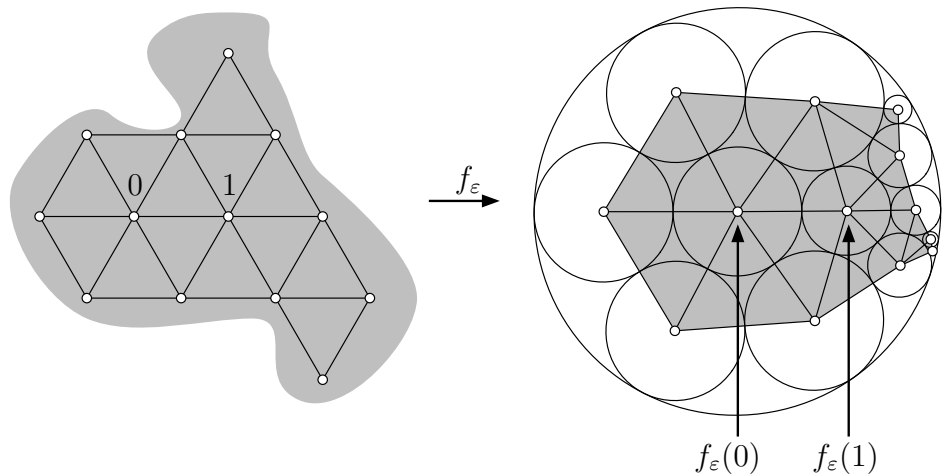

By applying the Theorem~\ref{CircTheorem}, we obtain a packing $\mathcal{P}$ on the Riemann sphere whose nerve is $H_\epsilon$ which is unique up to action by M\"obius transformations.  We may make this packing unique by forcing the packing to have the disk associated with infinity be the disk complement to the unit disk on the Riemann sphere, have the disk associated with $0$ be centered at $0$, and have the disk associated with $1$ be centered on a positive real.

By extending linearly over faces of $H_\epsilon$ contained within $\Omega$, one may now obtain a function $f_\epsilon$ which sends a subdomain of $\Omega$ to a subdomain of $\D$, the unit disk.  The Rodin--Sullivan Theorem is now the following.

\begin{theorem}[Rodin--Sullivan Theorem \cite{RodinSullivan}]\label{RSThm}
The function $f_\epsilon$ converges pointwise to the Riemann map from $\Omega$ to $\D$ fixing $0$ and sending $1$ to the positive real axis.
\end{theorem}

Due to this result, circle packings are often thought of as providing a discrete analogue to the Riemann mapping theorem.  For our discussion, we will now adopt this point of view, and use the circle packing as a method of providing a discrete way to assign a conformal structure to a triangulation.

To make connections with discrete holomorphic functions on quadrangulations, we now show how to convert a circle packed triangulation of the plane into an orthogonal quadrangulation.

To each triangle $T = [z_1z_2z_3]$, we associate the incenter $i_T$ which is the mutual intersection points of all interior angle bisectors of the triangle.  Let $c_1, c_2, c_3$ be the three points on the edges of $T$ which intersect the circle associated to the vertices of the circle packing opposite to the $z_k$ of the same index.  Then, the quadrilaterals $[z_kc_{k-1}i_Tc_{k+1}]$ are orthogonal for all $k = 1,2,3$.  This is illustrated in Figure~\ref{TriToQuad}.

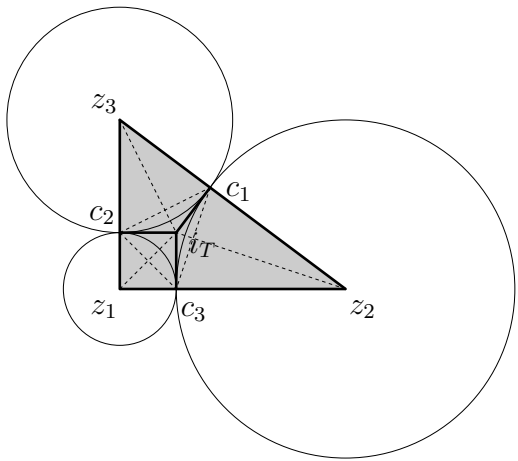
\begin{figure}[ht!]
	\labellist
	\pinlabel $z_1$ [rt] at 85 110
	\pinlabel $z_2$ [lt] at 210 110
	\pinlabel $z_3$ [rb] at 85 210
	\pinlabel $i_T$ [t] at 130 150
	\pinlabel $c_1$ [lb] at 140 160
	\pinlabel $c_2$ [r] at 85 155
	\pinlabel $c_3$ [t] at 125 110
	\endlabellist
	\includegraphics[width=3in]{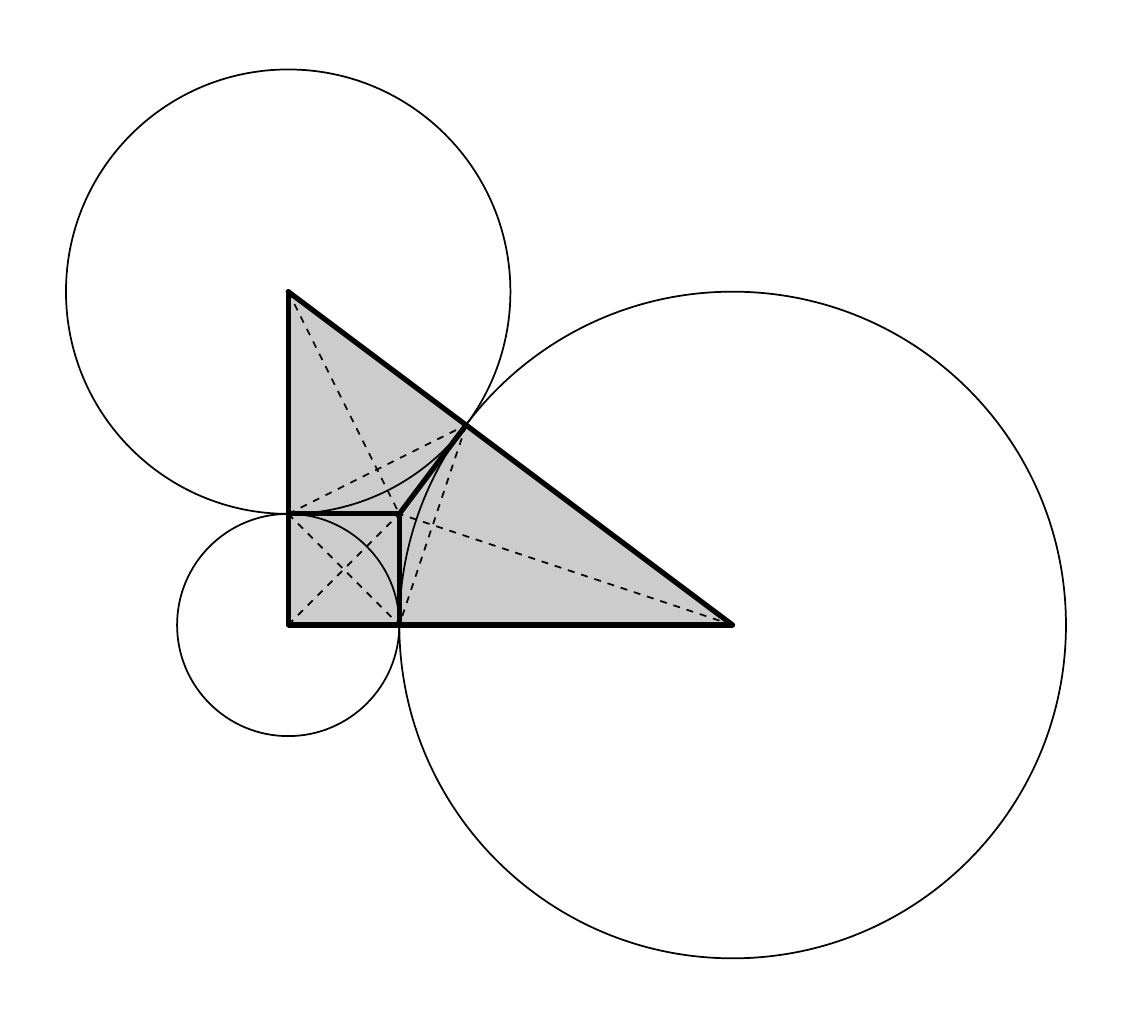}
	\caption{The procedure to convert a circlepacked triangle into a collection of orthogonal quadrilaterals.}
	\label{TriToQuad}
\end{figure}

If we wish to now apply our theorem to harmonic functions on such a discrete unifomization of the discrete Riemann surface, it will not be sufficient to simply have an orthogonal quadrangulation, we must have additional local control on the geometry of each of the faces.  We may do so by restricting our attention to triangulations of bounded degree and appealing to the Ring Lemma.

\begin{lemma}[Ring Lemma \cite{RodinSullivan}]\label{RingLemma}
	There is a constant $r_n$ depending only on $n$ such that if $n$ circles surround the unit disk (all disks cyclically tangent with each other and non-overlapping) then each circle has radius at least $r_n$.
\end{lemma}

The Ring Lemma only applies to interior vertices, thus to obtain the desired geometric control we must remove the quadrilaterals obtained from boundary triangles.  These results now prove the following circle packing version of the main theorem as a corollary.

\begin{corollary}\label{PackCor}
	Suppose $\{Q_n\}$ is a sequence of quadrilateral lattices approximating $\D$ obtained as above by circle packing triangulations with maximum degree uniformly bounded in $n$.  Then, the solution to the Dirichlet problem on $Q_n$ converges uniformly to the solution of the Dirichlet problem on $\D$.
\end{corollary}

\begin{proof}
Since the $\{Q_n\}$ were obtained by taking an orthogonal quadrangulation associated to circlepacked triangulations, we are in the setting above, hence we know that the $\{Q_n\}$ is a sequence of orthogonal quadrilateral lattices approximating $\D$.  By the Ring Lemma (Lemma~\ref{RingLemma}) there are uniform upper and lower bounds on the angles in the corners of each triangle, and thus on the edge ratios of the triangle.  This can be seen to imply $K$-roundness of the quadrilaterals by a computation in coordinates.
\end{proof}

One important fact to note about this corollary is to observe that the condition that the quadrilateral lattices approximate $\D$ includes within it the requirement that the diameter of the largest circle in the packing tends to zero.

From our motivation of random planar maps, the restriction to triangulations of bounded degree may seem severe, and indeed from the point of view of various combinatorial bijections (as in \cite{LeGallRev,BuziosNotes,HCSheff}) the condition is not simple to enforce, however it is believed that such a degree bound does not change the resulting scaling limit.  We now provide an explicit model of a random surface to which this theorem can be applied.

\begin{figure}
	\tiny
	\labellist
	\pinlabel $0$ [tc] at 29 5 
	\pinlabel $1$ [tc] at 61 5 
	\pinlabel $2$ [tc] at 93 5
	\pinlabel $X_t$ [l] at 100 17
	 
	\pinlabel $0$ [tc] at 169 5 
	\pinlabel $1$ [tc] at 201 5 
	\pinlabel $2$ [tc] at 233 5
	\pinlabel $X_t$ [l] at 240 17
	
	\pinlabel $0$ [tc] at 309 5 
	\pinlabel $1$ [tc] at 341 5 
	\pinlabel $2$ [tc] at 373 5
	\pinlabel $X_t$ [l] at 380 17
	
	\pinlabel $0$ [tc] at 509 5 
	\pinlabel $1$ [tc] at 541 5 
	\pinlabel $2$ [tc] at 573 5
	\pinlabel $X_t$ [l] at 580 17
	
	\pinlabel $0$ [r] at 20 20 
	\pinlabel $1$ [r] at 20 52 
	\pinlabel $2$ [r] at 20 84
	\pinlabel $Y_t$ [bc] at 26 100
	
	\pinlabel $0$ [r] at 160 20 
	\pinlabel $1$ [r] at 160 52 
	\pinlabel $2$ [r] at 160 84
	\pinlabel $Y_t$ [bc] at 166 100
	
	\pinlabel $0$ [r] at 300 20 
	\pinlabel $1$ [r] at 300 52 
	\pinlabel $2$ [r] at 300 84
	\pinlabel $Y_t$ [bc] at 306 100
	
	\pinlabel $0$ [r] at 500 20 
	\pinlabel $1$ [r] at 500 52 
	\pinlabel $2$ [r] at 500 84
	\pinlabel $Y_t$ [bc] at 506 100
	
	\endlabellist
	\includegraphics[width=6in]{DrivingFunction}
	\caption{An illustration of the notion of driving function topology as discussed in \cite{HCSheff}---a more in depth discussion can be found there.  The type of random surface under consideration is the uniform random measure on quadrangulations with $n$ faces where each face has either the diagonal between the white vertices or the diagonal between the black vertices added so that the induced white black subgraphs form a pair of dual trees.  In the above diagram, the trees are show in solid black lines, while the edges of the quadrangulation are shown dotted.  Given such a surface along with a choice of starting edge of the quadrangulation and orientation, there is a unique curve which visits each triangle of the subdivided quadrangulation exactly once which separates the white tree from the black tree.  As this path crosses each edge of the quadrangulation, the \emph{driving function} is obtained by keeping track of the distance to the root in the black tree, $X_t$, and the distance from the root in the white tree, $Y_t$.  This procedure produces a random walk in the non-negative quadrant of $\Z^2$ starting and ending at $(0,0)$ which can be seen to be in a bijective correspondence with the original map.  Notion of the driving function topology is the topology induced on marked quadrangulations induced by the $L^\infty$ (or some other choice of norm) topology of the driving functions.}
	\label{DrivingFig}
\end{figure}

\subsection{A bounded degree random surface}

One of the simplest to state models of these random surfaces is provided by a form of a classic bijection of Mullin \cite{Mullin} and further expanded by Bernardi \cite{Bernardi} and Sheffield \cite{HCSheff}.  The essential ingredient in these models is that one may produce a random surface by gluing two independent uniform random planar trees along their boundary cycles.  We will use the point of view that one may consider the \emph{driving function topology} as a reasonable and technically convenient topology to discuss the convergence of such random models \cite{HCSheff}.  A quick discussion of this topology is given in Figure~\ref{DrivingFig}. 

We produce here a model which converges to the same limit in the driving function topology as the model in \cite{HCSheff} (in the case $p=0$), while having uniformly bounded degree.

Given an integer $n$, we will construct a triangulation $T_n$ in the following way.  Start with three cycles of length $4n+2$, with marked origin, which we will refer to as the red cycle, green cycle, and blue cycle.  These cycles will be connected together in the manner indicated on the left of Figure~\ref{WeldFig}.  

\begin{figure}[ht!]
	\hskip 0.3in
	\tiny
	\labellist
	\pinlabel $\textrm{Red}$ [r] at 15 112
	\pinlabel $\textrm{Green}$ [r] at 15 97
	\pinlabel $\textrm{Blue}$ [r] at 15 82
	\endlabellist
	\includegraphics[width=4in]{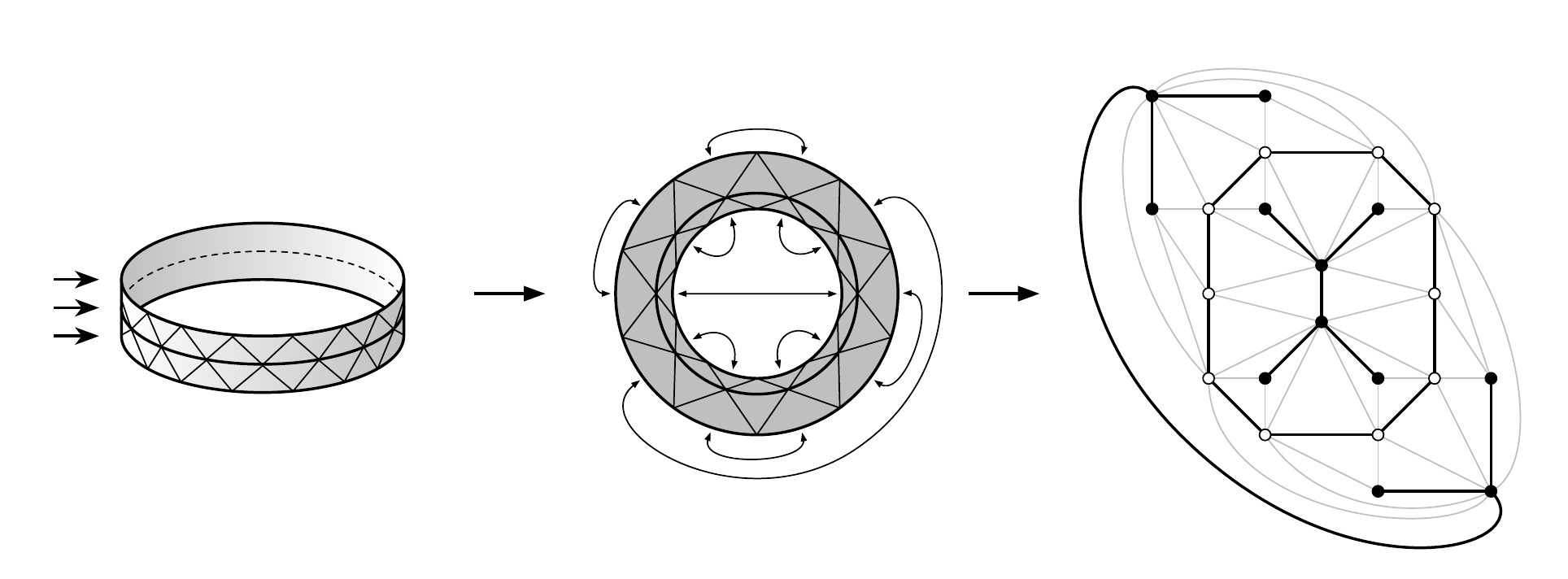}
	\caption{An illustration of the process of creating a surface of bounded degree from a pair of binary trees.  Start with three cycles, labeled red, green, and blue arranged on a cylinder as indicated on the left hand side, with the indicated edges.  We may view this cylinder as an annulus, as in the second figure, and consider a uniformly random way to glue the edges to each other on both the red and blue cycles to form independent random trees (the gluing is indicated with the arrows).  By welding the indicated edges, one gets the planar graph indicated in the third image, where the edges that were on the cycles are shaded more darkly than the edges interconnecting the cycles.  The vertices that were on the red and blue trees have been colored black, and the vertices of the cycle white.}
	\label{WeldFig}
\end{figure}

Given two leaf-rooted planar binary trees (trees where every node is either a leaf of degree one, or an internal node of degree three) chosen uniformly at random amongst the set of all such trees with $n$ internal nodes, we produce a surface as follows.  First note that the cycle of edges obtained by exploring counterclockwise around such a tree has length $2n+2$.  Thus, we may take the first random tree and use it to weld the red cycle (starting at their roots), and the other tree may then be used to similarly weld the blue cycle.  This process is illustrated in Figure~\ref{WeldFig}.  The result of this welding is a random triangulation of the sphere without loops or parallel edges.  The resulting triangulation has bounded degree since the two trees were of bounded degree.  Indeed, the fact that the we restricted to binary trees which have degree either one or three implies that every red and blue node has degree either three or nine, whereas every green node has exactly degree six. 

\begin{figure}[ht!]
	\includegraphics[width=2.12in]{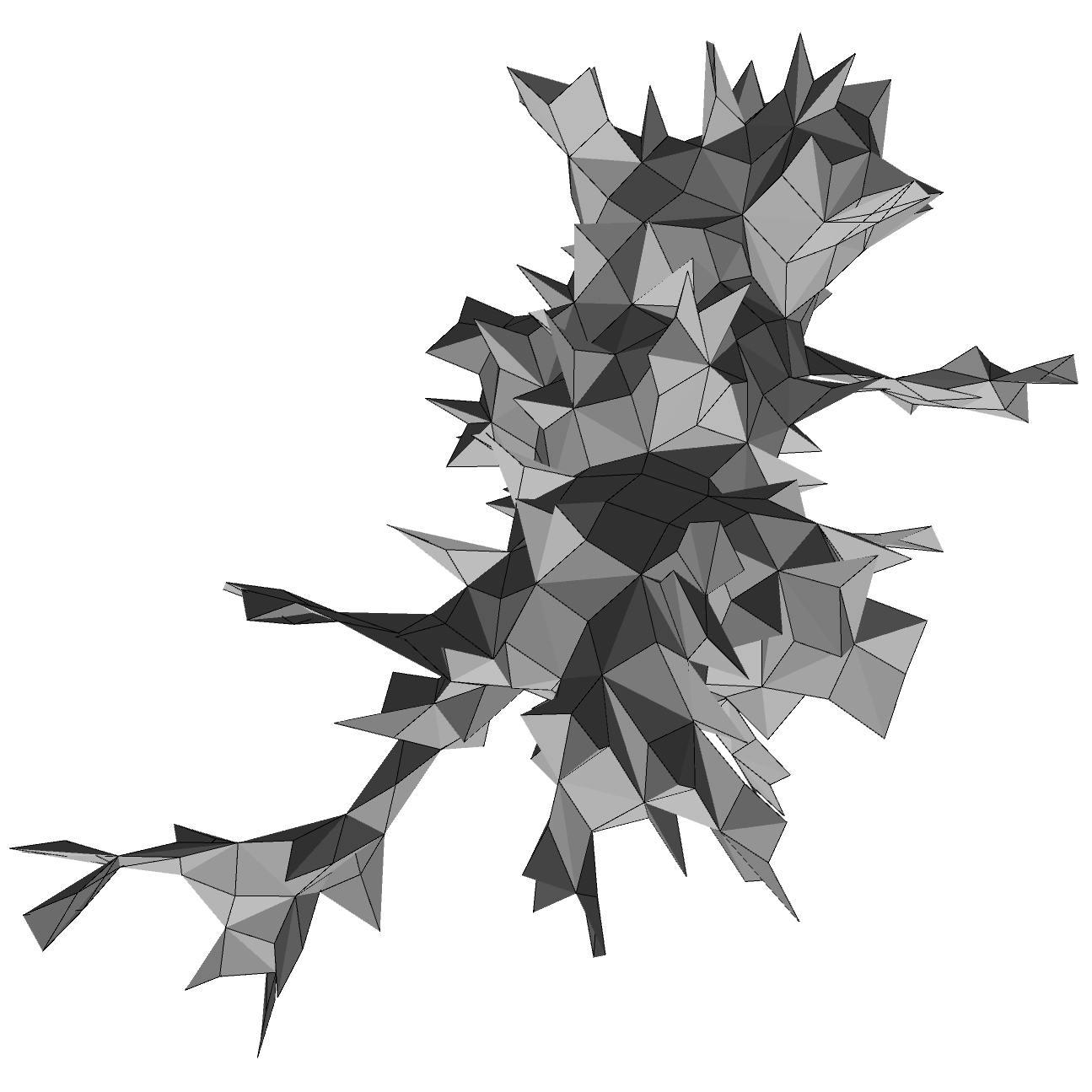}
	\includegraphics[width=2.12in]{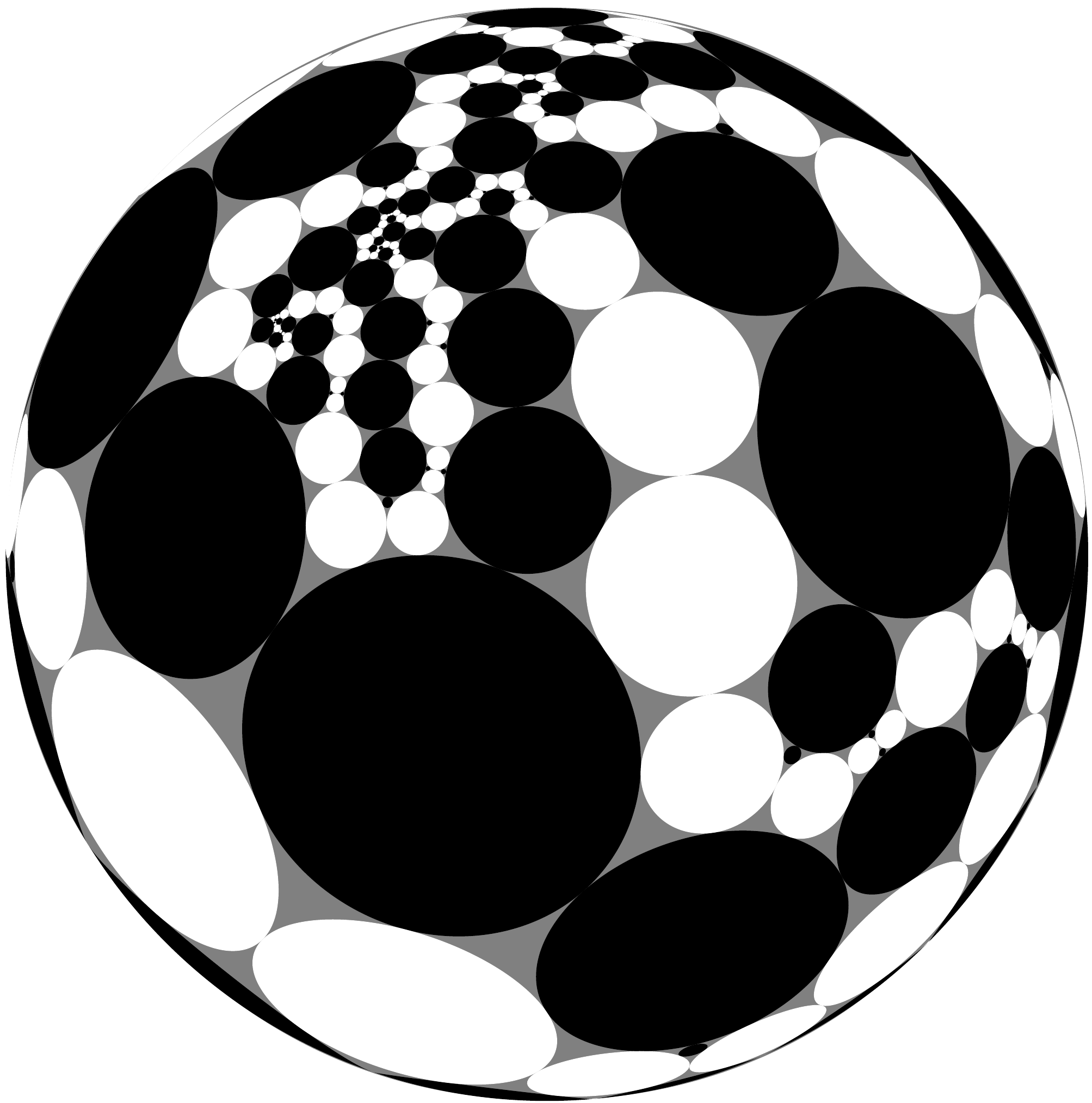}
	\includegraphics[width=2.12in]{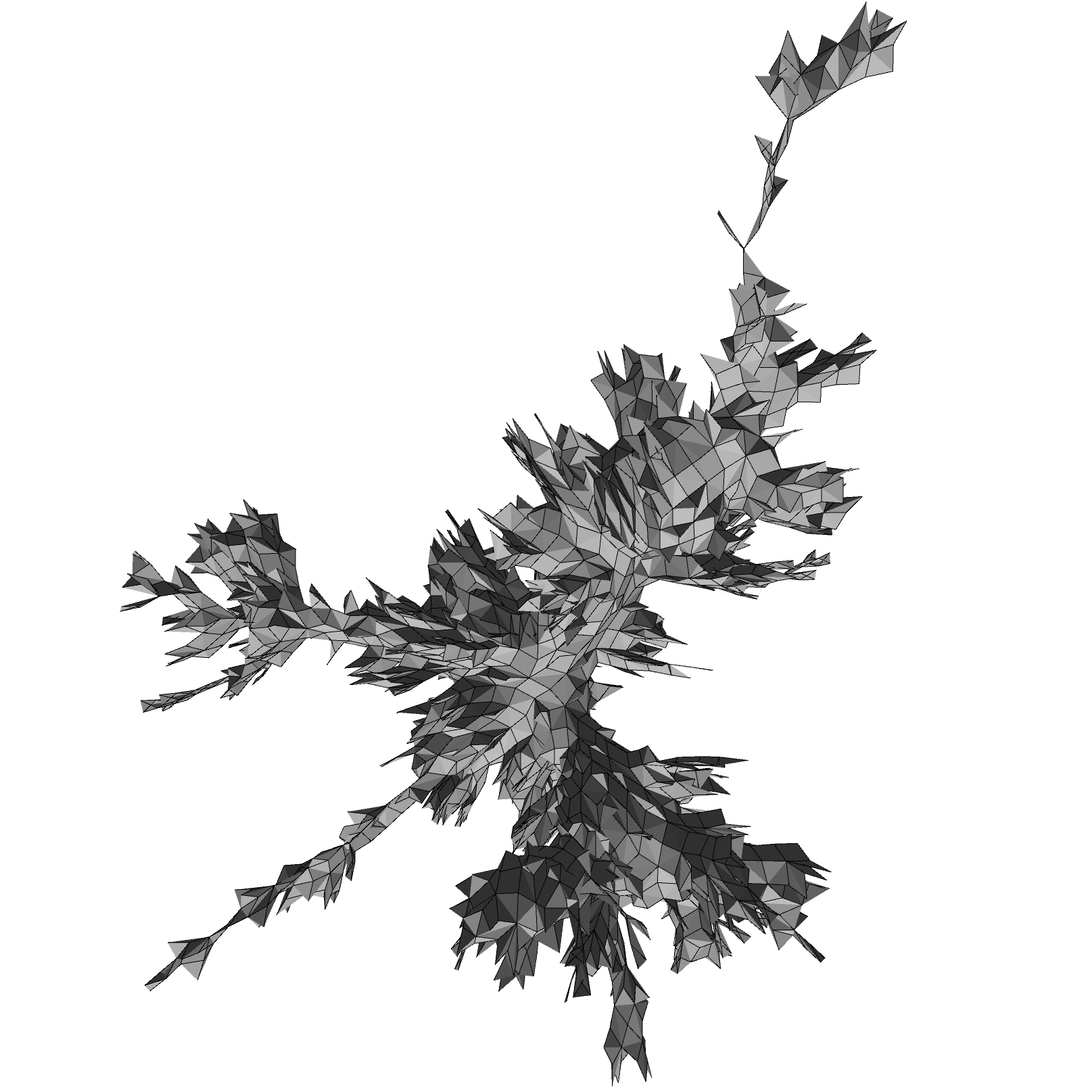}
	\caption{On the left is a small example of the type of random surface discussed above for $n=80$ ($644$ quadrilaterals) drawn approximately embedded in $\R^3$ using the author's own implementation of \cite{KKDraw}.  The center shows the same surface circle packed onto the sphere with the nodes on the trees colored black, and the nodes on the green cycle colored white.  The right image shows a much larger sample from the process with $n=200$ ($4004$ quadrilaterals). The circle packing was produced using CirclePack \cite{CPSoft}.}
	\label{SurfFig}
\end{figure}

Consider the following natural analog of the driving function topology.  For each edge of the green cycle there is a unique red vertex and blue vertex on the two triangular faces adjacent to the green edge.  Thus, for each green edge starting with the root edge, we may keep track of the distance of the adjacent red vertex in the red tree from the root red vertex (call this process $X_t$) and similarly the distance of the adjacent blue vertex from the root blue vertex (call this process $Y_t$).  The process $X_t$ is the contour process of the red rooted planar binary tree, and $Y_t$ the contour process of the blue rooted planar binary tree.  These trees were chosen independently and uniformly, thus to understand the process $(X_t,Y_t)$ we need only understand $X_t$.  However, the contour process of a uniformly random rooted planar binary tree (as it may be written as a conditioned Galton--Watson tree) is known to converge to a Brownian excursion \cite{CRT}.  Thus the process $(X_t,Y_t)$ converges to a Brownian excursion in the upper right quadrant, matching the $p=0$ model from \cite{HCSheff}.  This indicates that it is in the same universality class in the driving function topology and thus it is natural to conjecture that this model should tend to a $\sqrt{2}$-$\mathrm{LQG}$ \cite{QLE,CWSheff,HCSheff}. 

As a final note: these surfaces may be easily sampled via R\'emy's algorithm for constructing planar leaf-rooted binary trees \cite{Remy} where a uniformly random leaf-rooted planar binary tree with $n+1$ internal vertices may be generated from one with $n$ internal vertices by selecting a uniformly random side of an edge and attaching a new leaf at that point (see Figure~\ref{RemyFig}).

\begin{figure}[ht!]
	\includegraphics[width=4in]{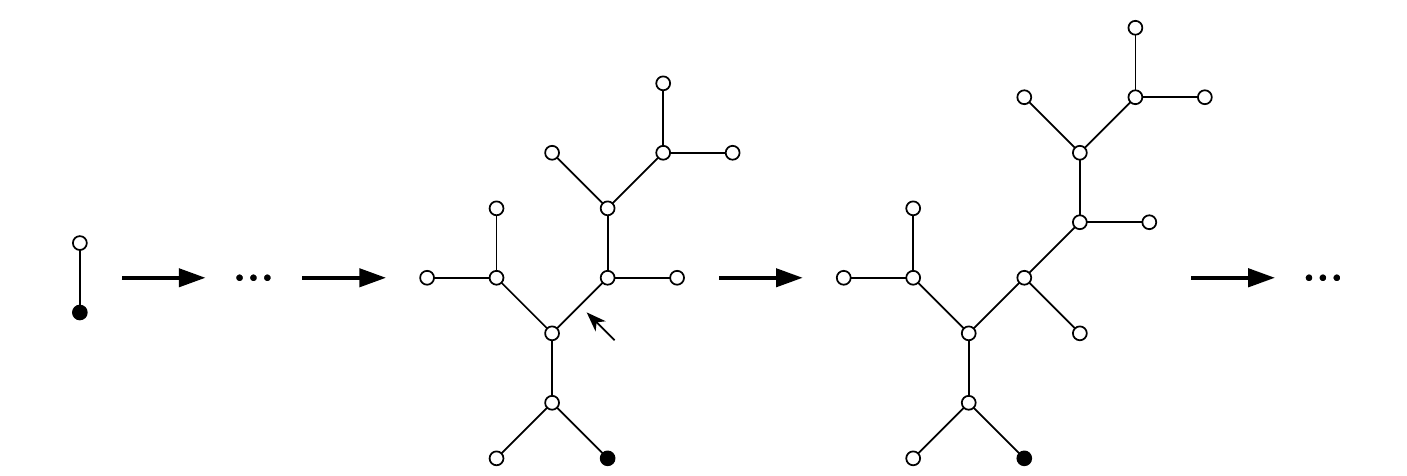}
	\caption{R\'emy's algorithm for constructing uniformly random planar leaf-rooted binary trees.  At each time step, a uniformly random side of an edge is selected and a new leaf is attached at that point.  The trees are considered as embedded trees modulo orientation preserving homeomorphisms of the plane.}
	\label{RemyFig}
\end{figure}

 A practical implementation of this algorithm can be found in \cite[Section 7.2.1.6]{Knuth}.  Additionally, this algorithm lifts to a natural growth procedure for the random surfaces described above by simply applying the growth step independently in each tree.  

One may use these trees to produce a uniformly $K$-round quadrangulation of the unit disk $\D$ in $\C$ as follows.  Pick three points uniformly at random from the surface and take the circle packing which sends these three points to circles centered on $0$, $1$, and $\infty$ in $\C$.  Convert the triangulation to an orthogonal quadrangulation as described in Section~\ref{MotSec}.  Restrict this packing to only those quadrilaterals which stay entirely within $\D$.  The resulting quadrangulation is orthogonal and uniformly $K$-round and thus by Corollary~\ref{PackCor}, we will have convergence in  Theorem~\ref{MainTheorem} as long as the resulting quadrangulations approximate $\D$.  This is currently unknown as it is currently unknown if the largest disk in such a packing will tend to zero in the limit (many such questions are open \cite{LeGallRev, HCSheff}), however conditional on this, we have convergence.

\section{Discrete holomorphicity}\label{DiscHoloSec}

Before proving our result, we first review the linear theory of discrete holomorphicity and harmonicity as provided by Skopenkov \cite{Skopenkov}, which we adopt in this paper.  Proofs of these results will be omitted and may be found in the original paper.  The main definition of discrete holomorphicity on a quadrilateral lattice is provided by a discrete form of the Cauchy-Riemann equations.

\begin{definition}
	A function $g:Q^* \rightarrow \C$ is \emph{discrete holomorphic} if for every face $f=[z_1z_2z_3z_4]$ we have
	\[
		\frac{g(z_1)-g(z_3)}{z_1-z_3} = \frac{g(z_2)-g(z_4)}{z_2-z_4}.
	\]
	A function $h:Q^* \rightarrow \R$ is \emph{discrete harmonic} if it is the real part of a discrete holomorphic function.
\end{definition}

It is worth making a few comments on this definition before continuing.  Note that the definition of discrete holomorphic may be viewed simply as a natural discretization that the difference quotient defining a single complex derivative is independent of the direction in which the difference is taken.  However, this is \emph{not} the same as making the definition that the difference quotient is independent at the \emph{vertices} of our quadrangulation as such a definition would either be over-determined (if all edges were used), or asymmetric under rotations of the lattice (if only a pair of edges were used).  While this definition does seem to provide a clean theory, the fact that the derivative is defined on faces rather than edges means that the discrete derivative is not discrete holomorphic in the same sense.  Discussion of this choice, in the special case of the square lattice, may be found in \cite{Duffin,SmirDisc}.    Additionally, note that the definition of discrete harmonicity is also not the standard definition that each value is the average of its neighbor's values on the graph, however we will see that it may be written in a similar form.

Finally, we should point out the natural role that the bipartition of $Q^* = Q^\bullet \sqcup Q^\circ$ plays in the theory.  Let $g$ be a discrete holomorphic function and $c\in \C$ be any complex number.  Then the function
\[
\tilde g(z) \ceq \begin{cases}
g(z) & z \in Q^\bullet, \\
g(z) + c & z \in Q^\circ,
\end{cases}
\]
is also discrete holomorphic as the definition only requires a relationship of the differences of the function on diagonal opposed vertices.  One may show that this is the only such ambiguity in the sense that if $g$ is defined on $Q^\bullet$ and may be extended to a holomorphic function on all of $Q^*$, then the extension is unique up to additive constant on $Q^\circ$.  This ambiguity leads to many results being best stated by restricting to one color of vertex alone (often black in this paper).

In the continuum, harmonicity may be described via the minimization of the \emph{Dirichlet energy}:
\[
E_\Omega(u) \ceq \int_\Omega |\nabla u|^2 \; dxdy.
\]
In the discrete such a definition may be given as well.  For a face $f=[z_1z_2z_3z_3]$ of $Q$, we will let the \emph{discrete gradient} of a function $u:Q^* \rightarrow \R$ be the unique complex number $\nabla_Q u(f)$ such that
\[
\nabla_Q u(f) \odot (z_3-z_1) = u(z_3) - u(z_1) \text{ and } \nabla_q u(f) \odot (z_4-z_2)  = u(z_4) - u(z_2).
\]
where we have adopted the notation that $(a+ib)\odot(c+id) = ac+bd$ denotes the dot product between two complex numbers regarded as vectors while we reserve $z\cdot w$ for their product as complex numbers.
We may now define the discrete Dirichlet energy.
\begin{definition}
	The \emph{discrete Dirichlet energy} is
	\[
	\E_Q(u) \ceq \sum_{f\in Q} |\nabla_Q u(f)|^2\cdot\area(f).
	\]
\end{definition}

As with the continuous Dirichlet energy, one may describe the condition of being discrete harmonic in term of the discrete Dirichlet energy.
\begin{lemma}[Convexity Principle $2.1$ \& Variation Principle $2.2$ \cite{Skopenkov}]
	The energy $E_Q(u)$ is a strictly convex functional on the affine space of functions with fixed values on $\bd Q$.  Moreover, $u$ is the unique minimizer for this energy with fixed values on $\bd Q$ if and only if $u$ is discrete harmonic.
\end{lemma}

In particular this immediately ensures existence of a unique solution to the Dirichlet problem.

\begin{corollary}[Existence and Uniqueness Theorem 1.1 \cite{Skopenkov}] 
	The Dirichlet problem on any finite quadrilateral lattice has a unique solution.
\end{corollary}

In the case of orthogonal lattices, which we assume throughout, the expression for the energy may be made more explicit:
\[
E_Q(u) = \frac{1}{2}\sum_{f=[z_1z_2z_3z_4] \in Q}\bigg[\frac{|z_2-z_4|}{|z_1-z_3|}(u(z_3)-u(z_1))^2+\frac{|z_1-z_3|}{|z_2-z_4|}(u(z_4)-u(z_2))^2\bigg].
\]

Again in analogy with the continuum, the condition of harmonicity may instead be expressed in terms of vanishing of a Laplacian which in this case takes the form:
\[
\Delta_Qu(z) \ceq -\frac{\bd E_Q(u)}{\bd u(z)} = \sum_{\substack{f=[z_1z_2z_3z_4]\in Q\\ z_1 = z}}\frac{|z_2-z_4|}{|z_1-z_3|}(u(z_3)-z(z_1)).
\]
When interpreting expressions containing the Laplacian in the future, it is worth while to note that this expression should be considered as already weighted by the area.  Thus, expressions which are integrals of the Laplacian in the continuum become sums of the discrete laplacian, and do not require further weighting, in contrast to the discrete gradient, which \emph{does} require weighting by the area, as in the definition of the discrete Dirichlet energy.  In this way, $\Delta_Qu(z)$ should be thought of as discretizing the $2$-form $\Delta u(x+iy) \;dxdy$ rather than $\Delta u(z)$.

From this expression, one may derive the following more geometric description of the Laplacian which will be used frequently in what follows.
\begin{lemma}[Lemma 3.3 \cite{Skopenkov}]\label{LapExpressLem}
	For each $z \in Q^*$, we have
	\[
	\Delta_Qu(z) = \sum_{\substack{f=[z_1z_2z_3z_4]\in Q\\ z_1 = z}} (i\nabla_Qu(f))\odot(z_2-z_4).
	\]
\end{lemma}

With these definitions, many of the familiar identities from the continuous theory may be translated to the discrete theory.  We will require the Green's identity and maximal principle.

\begin{lemma}[Lemma 3.7 \cite{Skopenkov}]\label{GreenLem}
	Let $Q$ be an orthogonal lattice and $u,v : Q^\bullet \rightarrow \C$ be arbitrary functions.  Then
	\[
	\sum_{z \in Q^\bullet} [u\nabla_Qv - v\nabla_Q u] = 0.
	\]
\end{lemma}

\begin{lemma}[Maximum Principle 3.5 \cite{Skopenkov}]\label{MaxLem}
	Let $Q$ be an orthogonal lattice and let $u:Q^* \rightarrow \R$ be discrete harmonic.  Then
	\[
		\max_{z \in Q^*} u(z) = \max_{z \in Q^* \cap \bd Q} u(z) \ \text{ and } \ \max_{z \in Q^\bullet} u(z) = \max_{z \in Q^\bullet \cap \bd Q} u(z).
	\]
\end{lemma}

\section{Geometric preliminaries}

Much of the results presented in this paper will rest on our ability to compare the geometry of paths within the lattice to near-by continuous paths in a uniform way.  We will first require a pair of additional local estimates in the form of the lemma below that tells us that essentially all of the local geometry of a single quadrilateral in a $K$-round quadrilateral lattice is controlled by $K$.

\begin{lemma}\label{MoreUnifLem}
	If a quadrilateral $f = [z_1z_2z_3z_4]$ is $K$-round then the lengths of the edges are bounded below by $\diam(f)/(2K)$ and the lengths of diagonals are bounded below by $\diam(f)/(4K^2)$.
\end{lemma}
\begin{proof}
	Since $f$ is a quadrilateral, the diameter is achieved between a pair of vertices.  These vertices are either connected by an edge, or a pair of edges.  If they are connected by an edge, then the lengths of all edges are are bounded below by $\diam(f)/K$ by definition of $K$-round.  If they are connected by two edges, then at least one edge has length at least $\diam(f)/2$, and thus we obtain a bound of $\diam(f)/(2K)$.
	
	Now consider the diagonal $z_1,z_3$ of $f$.  By the law of sines, we know that 
	\[
	|z_1-z_3| = \frac{\sin(\angle z_1z_2z_3)}{\sin(\angle z_2z_3z_1)}|z_1-z_2| \ge \frac{|z_1-z_2|}{2K} \ge \frac{\diam(f)}{4K^2}.
	\]
\end{proof} 

Additionally, we may control the area of such quadrilaterals in terms of their diameter.
\begin{lemma}\label{AreaLem}
	There exists a $C_K$, depending only on $K$, so that if $f = [z_1z_2z_3z_4]$ is s $K$-round quadrilateral, then $\area(f) \ge C_K \diam(f)^2$.
\end{lemma}
\begin{proof}
	Any quadrilateral, even when non-convex, may be split into two disjoint triangles by cutting along one of the diagonals.  Each triangle has two edges which are edges of the original quadrilateral which by Lemma~\ref{MoreUnifLem} are at least $\diam(f)/(2K)$ in length.  The angle between these two edges is at least $2\pi/K$.  Thus the area of each triangle is at least $\frac{\diam(f)^2}{8K^2}\sin(2\pi/K)$, proving the result.
\end{proof}

The primary significant geometric fact we will require is the following result, which allows us to approximate any given rectifiable curve with a union of quadrilaterals so that the sum of the diameters of the quadrilaterals is not bigger than some constant times the length of the given curve.

\begin{proposition}\label{DiamProp}
	There exists a $C_K$, depending only on $K$, such that the following holds.  Let $Q$ be a $K$-round lattice and $\gamma:[0,1]\rightarrow\C$ be a rectifiable closed loop with $\diam(\gamma) > 2M(Q)$.  Then
	\[
	\sum_{\substack{f \in Q\\ f \cap \gamma \neq \emptyset}} \diam(f) \le C_K \ell(\gamma)
	\]
	where $\ell(\gamma)$ is the length of $\gamma$.
\end{proposition}

To prove the above proposition, one would want to show that each face $f$ which intersects $\gamma$ needs to contain a length of curve comparable to the diameter of $f$.  This is, however, false since a quadrilateral with arbitrarily large diameter may intersect a rectifiable curve in arbitrarily small length since it may pass through the corner.  To circumvent this issue, we will make use of the following lemma on the local neighborhood of a single face.	Let $\hat f$ denote the sublattice of $Q$ consisting of the face $f$ and all faces $f'$ which share a vertex with $f$.  We will call such a sublattice the \emph{neighborhood} of $f$.  What we will show is that the even if a single face may intersect in arbitrarily small length, the neighborhood of that face must contain a length of the curve comparable to the diameter of the original face.

\begin{lemma}\label{IntLem}
There exists an $\epsilon_K$, depending only on $K$, such that the following holds.  Let $Q$ be a $K$-round lattice and $\gamma:[0,1]\rightarrow\C$ be a rectifiable closed loop with $\diam(\gamma) > 2M(Q)$.  If $f$ is any face such that $f \cap \gamma \neq \emptyset$ and , then
\[
\sum_{f' \in \hat f} \ell(\gamma \cap f') \ge \epsilon_K \diam(f).
\]
\end{lemma}
\begin{proof}
	First, note that because the faces of $Q$ are $K$-round, we know that all angles are bounded below by $2\pi/K$ and hence that no more than $K$ quadrilaterals may be incident on any single vertex.  Let $m$ be the minimal edge length of $f$.  Since $f$ is $K$-round, we know by Lemma~\ref{MoreUnifLem} that $m \ge \diam(f)/(2K)$.  Let $\hat m$ denote the minimal length of an edge incident on a vertex of $f$.  By repeated application of the definition of $K$-round, using the two facts above, we see that $\hat m \ge \diam(f)/(2K^{K+1})$.  Moreover, by another application of Lemma~\ref{MoreUnifLem}, we see that every diagonal incident on a vertex of $f$ is bounded below in length by $\hat m/(4K^2) \ge \diam(f)/(8K^{K+3})$.  Thus, we see that $\hat f$ contains the $\epsilon_K' \diam(f)$ neighborhood of the the \emph{corners} of $f$, for $\epsilon_K' = \frac{1}{8K^{K+3}}$.
	
	\begin{figure}[ht!]\label{NeighborhoodFig}
		\labellist
		\pinlabel $f$ [bl] at 47 65
		\endlabellist
		\includegraphics[width=2.5in]{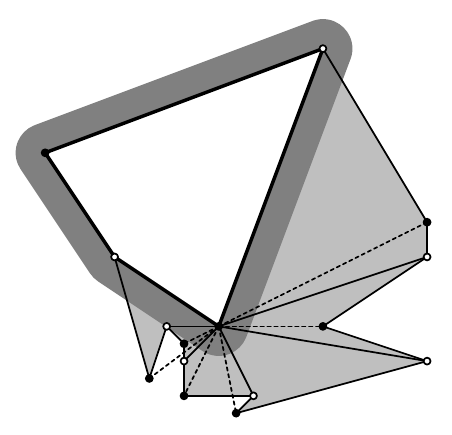}
		\caption{The portion of the neighborhood of a face $f$ (shown in white) in $Q$ incident on a single vertex.  The faces of the neighboring faces are shown in light grey with diagonals shown dashed.  Dark grey illustrates the $\epsilon_K \diam(f)$ neighborhood of $f$ which may be found in the union of the neighboring faces.}
	\end{figure}
	
	This may be extended to a neighborhood of the entirety of $f$ by invoking the definition of $K$-round directly and noting that since every angle is bounded below by $2\pi/K$, we know that the opposite edge of the any quadrilateral adjacent to any edge of $f$ must be at least $\epsilon_K \ceq \sin(2\pi/K) \epsilon_K'$ away, thus providing a full $\epsilon_K \diam(f)$ neighborhood of $f$.  This is illustrated in Figure~\ref{NeighborhoodFig}.
	
	Since $\diam(\gamma) \ge 2M(Q) \ge 2\epsilon_K \diam(f)$ we know that
	\[
	\diam\Big[\bigcup_{f' \in \hat f} \gamma \cap f'\Big] \ge \epsilon_K \diam(f),
	\]
	and thus, because the lengths must sum up to at least the diameter,
	\[
	\sum_{f' \in \hat f} \ell(\gamma \cap f') \ge \epsilon_K \diam(f).
	\]
\end{proof}

We may now prove the main geometric estimate.

\begin{proof}[Proof of Proposition~\ref{DiamProp}]
	By summing the result of Lemma~\ref{IntLem} over all faces intersecting $\gamma$ we see that
	\[
	\sum_{\substack{f \in Q\\ f \cap \gamma \neq \emptyset}} \diam(f) \le \frac{1}{\epsilon_K} \sum_{\substack{f \in Q}{f\cap \gamma \neq \emptyset}} \sum_{f'\in \hat f} \ell(\gamma \cap f').
	\]
	Since, as in the previous proof, there are at most $K$ faces incident on any single vertex, we know that $\hat f$ is composed of at most $4K$ faces, and thus that for any $f'$ there are at most $4K$ $f\in Q$ with $f' 
	\in \hat f$.  Thus,
	\begin{align*}
	\frac{1}{\epsilon_K} \sum_{\substack{f \in Q}{f\cap \gamma \neq \emptyset}} \sum_{f'\in \hat f} \ell(\gamma \cap f') & = \frac{1}{\epsilon_K} \sum_{f' \in Q} \ell(\gamma \cap f') \cdot \#\{f \in Q \ : \ f\cap \gamma \neq \emptyset, f' \in \hat f\} 
\\
& \le \frac{4K}{\epsilon_K} \sum_{f' \in Q} \ell (\gamma \cap f') = C_K \ell(\gamma).
\end{align*}
\end{proof}

\section{Energy estimates and equicontinuity}

The previous section provided a number of lemmas relating curves to discrete equivalents in the quadrilateral lattice.  In this section we will take this uniform geometric control and use it to provide control on the discrete Dirichlet energy of functions on the lattice, as well as control on the modulus of continuity of discrete harmonic functions.

\begin{definition}
	The \emph{semi-energy} of a function $u:Q^*\rightarrow \R$ along the path $w_0w_1\ldots w_m$ in black vertices of $Q^\bullet$ is 
	\[
	\hat E_{w_0 \ldots w_m}(u) \ceq \sum_{i = 1}^m |\nabla_Q u(f_i)|^2\cdot|w_i-w_{i-1}|
	\]
	where $f_i$ is the quadrilateral with diagonal $w_{i-1}w_i$.
\end{definition}

It is important to note that the definition of semi-energy takes as its input a chain of black vertices so that each consecutive pair of vertices are the diagonal of a face.  This means the semi-energy may equally well be regarded as a chain of faces with consecutive faces intersecting in $Q^\bullet$.

\begin{lemma}\label{PathEnergyLem}
	Let $u:Q^*\rightarrow \R$ be any function on any quadrilateral lattice and let $w_0w_1\ldots w_m$ be a path on the black vertices of $Q^\bullet$.  Then,
	\[
	\hat E_{w_0\ldots w_m}(u) \ge \frac{(u(w_m)-u(w_0))^2}{\ell(w_0\ldots w_m)},
	\]
	where
	\[
	\ell(w_0\ldots w_m) \ceq \sum_{j = 1}^m |w_j-w_{j-1}|
	\]
	is the Euclidean length of the path.
\end{lemma}
\begin{proof}
	By definition of the semi-energy and discrete Laplacian we have that
	\begin{align*}
	\hat E_{w_0 \ldots w_m}(u)  & = \sum_{i = 1}^m |\nabla_Q u(f_i)|^2\cdot|w_i-w_{i-1}| \\
	&\ge \sum_{i = 1}^m \frac{(u(w_i)-u(w_{i-1}))^2}{|w_i-w_{i-1}|^2}\cdot|w_i-w_{i-1}| \\
	&= \sum_{i = 1}^m \frac{(u(w_i)-u(w_{i-1}))^2}{|w_i-w_{i-1}|}.
\end{align*}
Now, by Cauchy--Schwartz, we see that
\begin{align*}
	|u(w_m)-u(w_0)|^2 & = \Biggl| \sum_{i=1}^m\frac{(u(w_i)-u(w_{i-1}))}{\sqrt{|w_i-w_{i-1}|}}\cdot\sqrt{|w_{i}-w_{i-1}|} \Biggr|^2\\
	& \le \Biggl[\sum_{i=1}^{m}\frac{(u(w_i)-u(w_{i-1}))^2}{|w_i-w_{i-1}|}\Biggr]\cdot\Biggl[\sum_{i=1}^{m}|w_{i}-w_{i-1}|\Biggr] \\
	& \le \hat E_{w_0\ldots w_m}(u)\cdot \ell(w_0\ldots w_m).
\end{align*}
Which is a rearrangement of the desired inequality.
\end{proof}

We will wish to use this lemma to obtain an estimate on the energy of a discrete harmonic function over the entire lattice in terms of the difference between the values of the function at a pair of points.  To do so, we will first apply the above lemma to estimate the amount of semi-energy ``at distance $r$'' from our pair of points.  In particular, fix a point $z\in\C$ and $r>0$, then let 
\[
\hat E_r^z(u) \ceq \sum_{\substack{f\in Q\\f \cap \bd B_r(z) \neq \emptyset}} |\nabla_Q u(f)|^2\cdot\diam(f).
\]
This provides the semi-energy at distance $r$.  By integrating over $r$ we will then obtain an estimate on the energy itself.  

Throughout the following set of arguments we will be given a pair of points $z$ and $w$ and we will always only consider balls centered on the point $(z+w)/2$.  We will thus suppress explicit dependence on the center point and write $\hat E_r(u)$ for $\hat E_r^{(z+w)/2}(u)$ and $B_r$ for $B_r((z+w)/2)$.

\begin{lemma}\label{SemiEnergyLem}
	There exists a $C_K$, depending only on $K$, such that the following holds. Let $Q$ be a $K$-round orthogonal lattice, $z,w$ be a pair of black vertices in $Q^\bullet$, and $u:Q^*\rightarrow \R$ be a discrete harmonic function. Take $r > |z-w| \vee M(Q)$ restricted to those $r$ with no vertices of $Q$ on the circle of radius $r$ about $(z+w)/2$.  Let
		\[
		\delta_r \ceq |u(z)-u(w)| - \max_{z',w' \in \bd Q \cap B_r} |u(z')-u(w')|.
		\]
		If $\delta > 0$ Then
		\[
		\hat E_{r}(u) \ge \frac{C_K\delta_r^2}{r}.
		\]
\end{lemma}
\begin{proof}
  Without loss of generality, assume $u(z) > u(w)$.  Let $Q_r$ be the lattice of all quadrilaterals which intersect $B_r$.  Note that $Q_r$ need not, in general, be connected and thus we let $Q_r^z$ be the component of $Q_r$ containing $z$ and $Q_r^w$ be the component of $Q_r$ containing $w$.  By the maximum principle (Lemma~\ref{MaxLem}) there exists a point $z' \in \bd Q_r \cap Q^\bullet$ with $u(z')> u(z)$ and a point $w' \in \bd Q_r \cap Q^\bullet$ with $u(w') < u(w)$.  

\emph{Case 1.}
Consider the case that $\bd Q_r \cap \bd Q = \emptyset$.  Since $z'$ and $w'$ are contained in $\bd Q_r$, we know there exists two faces $f_{z'}, f_{w'}$ which have $z'$ and $w'$ as corners which are moreover intersect $\bd B_r$ (otherwise all neighboring faces would be strictly contained in $B_r$ and thus the point could not be on $\bd Q$).  By following the arc of the circle $\bd B_r$ between those two faces, and taking the black edge of all faces on the arc, one obtains a connected set of black edges connecting $z'$ and $w'$.  Thus there is a path $w_0\ldots w_m$ of black edges connecting $z'$ to $w'$ made entirely of black edges contained in faces intersecting $\bd B_r$.  This case is illustrated in Figure~\ref{CircFig1}.

\begin{figure}[ht!]
	\labellist
	\pinlabel $z$ [bl] at 50 95
	\pinlabel $w$ [r] at 75 70
	\pinlabel $z'$ [l] at 100 135
	\pinlabel $w'$ [l] at 110 20
	\endlabellist
	\includegraphics[width=3in]{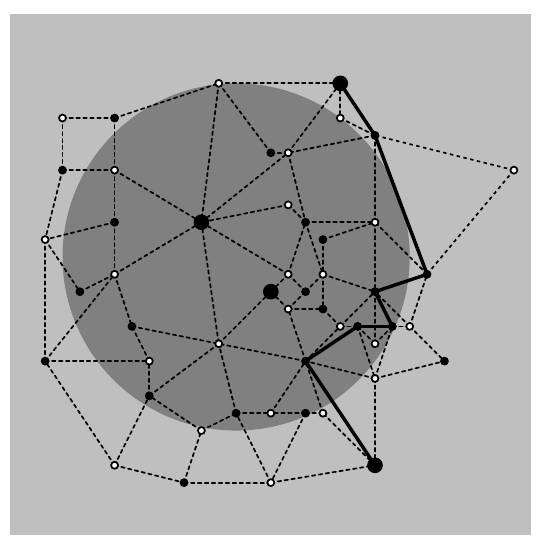}
	\caption{An illustration of \emph{Case 1}.  The light grey background denotes the domain formed by the union of faces in $Q$.  In this case the boundary is far from the set being considered.  The dark grey illustrates $B_r$ and the dashes lines and colored vertices show $Q_r$.  The bold points are as labeled, and the bold path indicates the diagonal edges between black vertices connecting $z'$ to $w'$ used to provide a lower bound on the energy of the boundary quadrilaterals.}
	\label{CircFig1}
\end{figure}
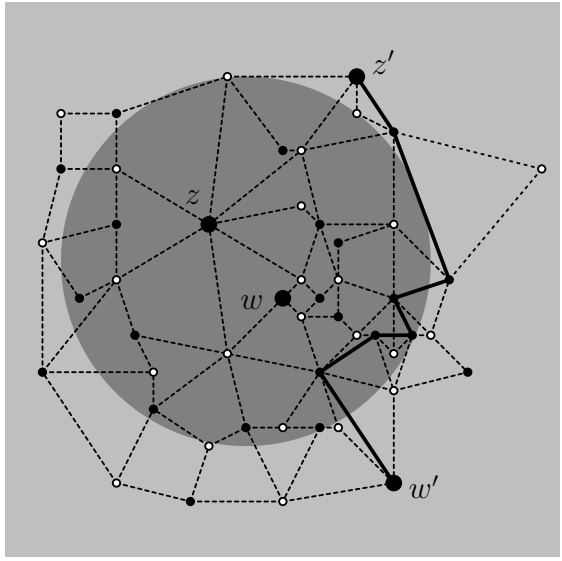

By Lemma~\ref{PathEnergyLem}, we know that
\[
	\hat E_r(u) \ge \hat E_{w_0\ldots w_m}(u)
 \ge \frac{(u(z')-u(w'))^2}{\ell(w_0\ldots w_m)} 
 \ge \frac{(u(z)-u(w))^2}{\ell(w_0\ldots w_m)} 
\ge \frac{\delta^2}{\ell(w_0\ldots w_m)}.
\]
By applying Lemma~\ref{DiamProp}, we see that
\[
\hat E_r(u) \ge \frac{\delta^2}{\ell(w_0\ldots w_m)} \ge \frac{\delta^2}{\sum_{\substack{f \in Q\\f \cap \bd B_r \neq \emptyset}} \diam(f)} \ge \frac{C_K\delta^2}{r}
\]

\emph{Case 2.}
Now, consider the case that $\bd Q_R \cap \bd Q \neq \emptyset$.  If there exists an arc of the circle $\bd B_r$ with the same properties of the previous case which stays inside $Q$, we are again done, so assume that $\bd B_r$ is split into multiple components.  First, assume that $z'$ and $w'$ are contained in $\bd Q_r \setminus \bd Q$  Let $C_{z'}$ be the arc which intersects a face which contains $z'$ and $C_{w'}$ be the similar choice for $w'$.  Let $z''$ be the vertex in $\bd Q_r \cap \bd Q$ at one of the endpoints of $C_{z'}$, and the same for $w''$.  This case is illustrated in Figure~\ref{CircFig2}.

\begin{figure}[ht!]
	\labellist
	\pinlabel $z$ [bl] at 50 95
	\pinlabel $w$ [r] at 75 70
	\pinlabel $z'$ [l] at 100 135
	\pinlabel $w'$ [l] at 45 10
	\pinlabel $z''$ [l] at 100 65
	\pinlabel $w''$ [l] at 5 60
	\endlabellist
	\includegraphics[width=3in]{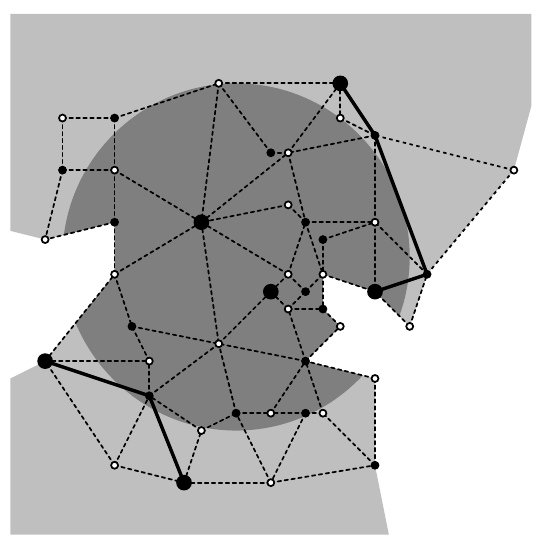}
	\caption{An illustration of \emph{Case 2} with the same interpretation as in Figure~\ref{CircFig1}.  In this case the boundary of $Q$ intersects $Q_r$ non-trivially.}
	\label{CircFig2}
\end{figure}
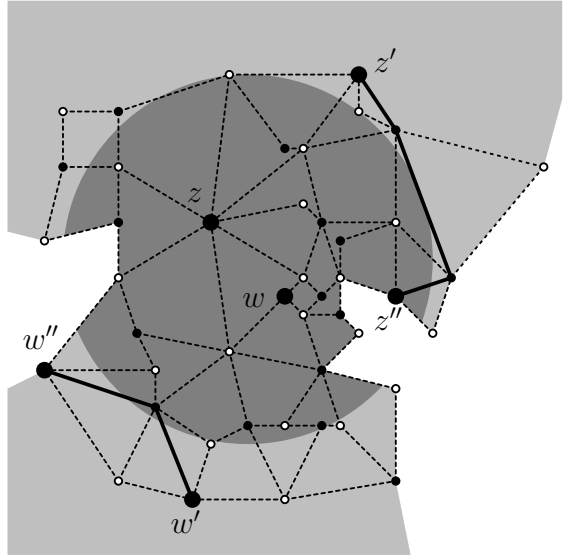

By definition of $\delta$,
\[
(u(z') - u(z'')) + (u(w'') - u(w')) \ge (u(z) - u(w)) - (u(w'') - u(z'')) \ge \delta,
\]
and thus one of $u(z') - u(z'')$ or $u(w'')-u(w')$ is greater than $\delta/2$.  Applying the same argument as the first case to those points yields the desired bound.

\emph{Case 3.}
Now assume one of $z'$ or $w'$ is contained in $\bd Q_r \setminus \bd Q$ and the other is not; we will assume, without loss of generality, that $z' \in \bd Q_r \setminus \bd Q$.  Take $z''$ as in the second case. This case is illustrated in Figure~\ref{CircFig3}.

\begin{figure}[ht!]
	\labellist
	\pinlabel $z$ [bl] at 50 95
	\pinlabel $w$ [r] at 75 70
	\pinlabel $z'$ [l] at 100 135
	\pinlabel $z''$ [l] at 100 65
	\pinlabel $w'$ [l] at 5 60
	\endlabellist
	\includegraphics[width=3in]{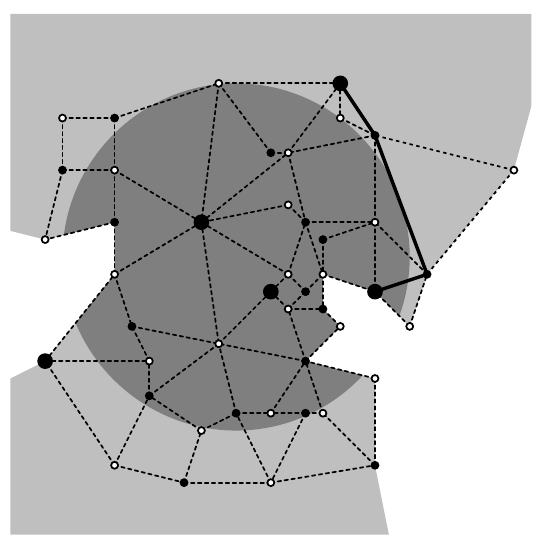}
	\caption{An illustration of \emph{Case 2} with the same interpretation as in Figure~\ref{CircFig1}.  In this case the boundary of $Q$ intersects $Q_r$ non-trivially, and the maximal value is attained on the boundary.}
	\label{CircFig3}
\end{figure}
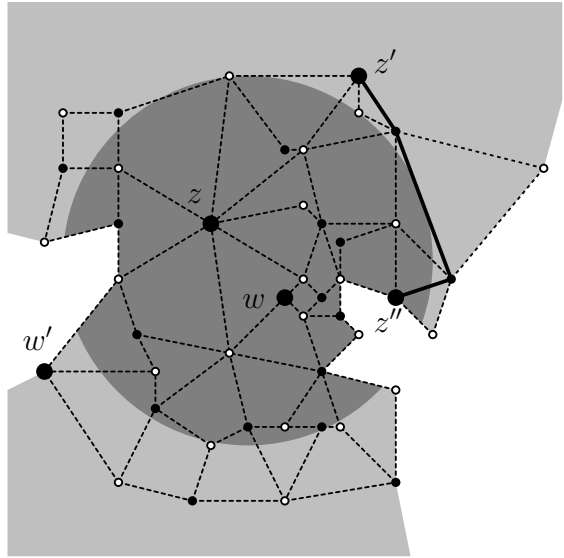

Then we have, by the definition of $\delta$, that
\[
u(z') - u(z'') = (u(z')-u(w'))-(u(z'')-u(w')) \ge (u(z)-u(w))-(u(z'')-u(w')) \ge \delta.
\]
The desired bound may again be obtained as in the first case.

\emph{Case 4.}
If neither $z'$ nor $w'$ are $\bd Q_r \setminus \bd Q$, then we have that
\[
0 > (u(z)-u(w))-(u(z')-u(w')) \ge \delta
\]
which is a contradiction.
\end{proof}

We now integrate this in $r$ to obtain the desired result.

\begin{proposition}\label{EquiProp}
There exists a $C_K$, depending only on $K$, such that the following holds. Let $Q$ be a $K$-round orthogonal lattice, $z,w$ be a pair of black vertices in $Q^\bullet$, and $u:Q^*\rightarrow \R$ be a discrete harmonic function on $Q^*$.  Then, for $R \ge |z-w| \vee M(Q)$
	\[
		|u(z)-u(w)| \le C_K E_{Q_R}(u)^{1/2}\log^{-1/2}\Bigl[\frac{R}{|z-w|\vee M(Q)}\Bigr] + \max_{z',w' \in \bd Q \cap B_R} |u(z')-u(w')|.
	\]
\end{proposition}
\begin{proof}
	Let 
	\[
	\delta_R \ceq |u(z)-u(w)| - \max_{z',w' \in \bd Q \cap B_r} |u(z')-u(w')|.
	\]
	If $\delta_R \le 0$, the desired result holds immediately, so assume $\delta_R > 0$.
	
	By Lemma~\ref{SemiEnergyLem} (noting that since $Q$ is finite, the bound applies for all but finitely many $r$), and by observing that $\delta_r > \delta_R$ for $r<R$,
	\begin{align*}
	\int_{|z-w| \vee M(Q)}^R \hat E_{r}(u) \ dr & \ge \int_{|z-w| \vee M(Q)}^R \frac{C_K\delta_r^2}{r} \\
	& \ge \int_{|z-w| \vee M(Q)}^R \frac{C_K\delta_R^2}{r} dr \\
	& = C_K\delta_R^2\log\Bigl[\frac{R}{|z-w|\vee M(Q)}\Bigr]
	\end{align*}
	
	Now, by examining the definition of semi-energy, and Lemma~\ref{AreaLem}, we see
	\begin{align*}
		\int_{|z-w| \vee M(Q)}^R \hat E_{r}(u) \ dr & \le \int_0^R \sum_{\substack{f\in Q \\ f \cap \bd B_r \neq \emptyset}} |\nabla_Q u(f)|^2\cdot\diam(f) \ dr \\
		& \le \sum_{\substack{f \in Q \\ f \cap B_r \neq \emptyset}} |\nabla_Q u(f)|^2\cdot \diam(f)^2 \ dr \\
		& \le C_K' \sum_{\substack{f \in Q \\ f \cap B_r \neq \emptyset}} |\nabla_Q u(f)|^2\cdot \area(f) \ dr \\
		& = C_K' E_{Q_R}(u).
	\end{align*}
	
	Combining these, substituting the definition of $\delta_R$, and solving for $|u(z)-u(w)|$ we obtain the following estimate:
	\[
		|u(z)-u(w)| \le C_K E_{Q_R}(u)^{1/2}\log^{-1/2}\Big[\frac{R}{|z-w|\vee M(Q)}\Big] + \max_{z',w' \in \bd Q \cap B_r} |u(z')-u(w')|.
	\]
\end{proof}


\section{Laplacian Approximation}
Unlike the equicontinuity result above, the Laplacian approximation lemma requires only very slight modifications from the proof presented in \cite{Skopenkov}.  For completeness we include the proofs here.  Our aim is to prove the following

\begin{proposition}\label{LapProp}
	There exists a $C_K$, depending only on $K$, such that the following holds.  Let $Q$ be a $K$-round quadrilateral lattice, and $R$ be square of side-length $r > M(Q)$ inside $\bd Q$.  Then, for any $g \in C^3(\C)$ we have
	\[
	\Big|\sum_{z \in R \cap Q^\bullet} [\Delta_Q(g|_{Q^*})](z) - \int_R\Delta g \; dx dy\Big| \le C_K \Big(M(Q)r\max_{z \in R}|D^2g(z)|+r^3\max_{z \in R}|D^3g(z)|\Big).
	\]
\end{proposition}

The proof of this proposition proceeds as follows.  Take an arbitrary function $g \in C^3(\C)$, and without loss of generality assume that $R$ is centered at $0$.  Expand $g$ as
\[
g(z) = a_0 + a_1\Re z + a_2 \Im z + a_3 \Re z^2 + a_4 \Im z^2 + a_5 |z|^2 + \tilde g(z)
\]
where $D^k\tilde g(0) = 0$ for $k = 0,1,2$.

As long as Proposition~\ref{LapProp} can be proven for each individual term in the above expansion the general case follows.  The cases $g(z) = 1, \Re(z), \Im(z)$ all follow immediately since all three of those functions are both harmonic and discrete harmonic, and thus the left hand term of Proposition~\ref{LapProp} is always zero.  The remaining terms require more work.  

We will first prove the cases $\Re (z^2)$, $\Im (z^2)$, and $|z|^2$ as follows.  By direct computation,
\[
\int_R \Delta |z^2| \; dx dy = \int_R 4 \; dxdy = 4\,\area(R)
\]
and
\[
\int_R \Delta \Re (z^2) \; dxdy =  \int_R \Delta \Im (z^2) \; dxdy = \int_R 0 \; dxdy = 0.
\]
Thus, the result will follow if we can show that the sum of the discrete Laplacian computes a discrete equivalent of the same term.  Indeed we will see the agreement is exact for all quadrilaterals with both black vertices contained within $R$, and the error can be bounded by the total area of quadrilaterals with one black vertex in $R$ and the other vertex not as illustrated in Figure~\ref{RQFig}

\begin{figure}[ht!]
	\includegraphics[width=3in]{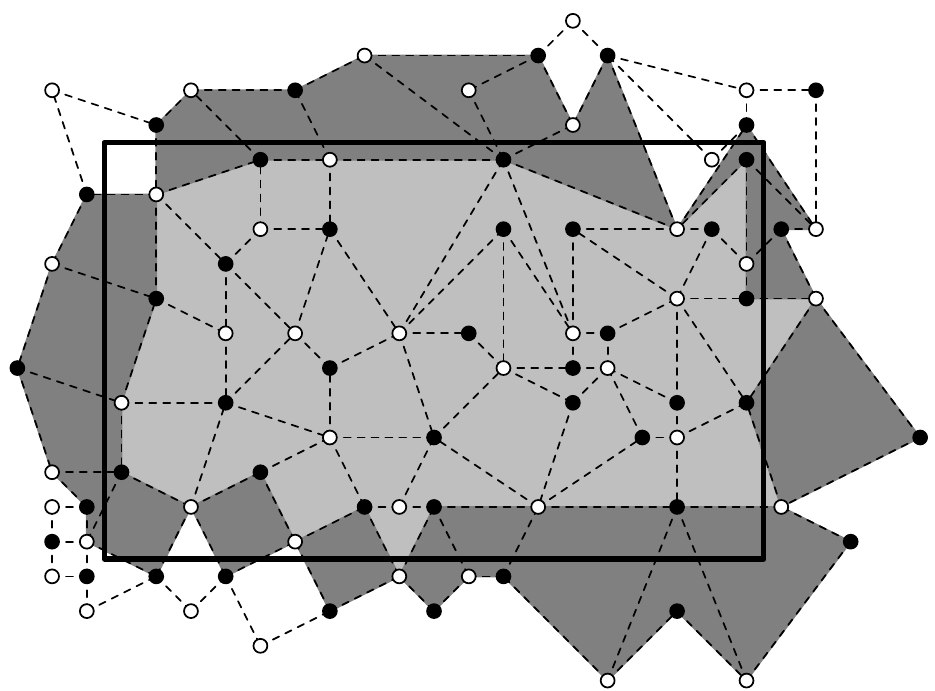}
	\caption{An illustration of the quadrilaterals considered in the proof of Proposition~\ref{LapProp}.  The light grey quadrilaterals are those with both black vertices contained within the rectangle $R$ (show by the bold lines), whereas the dark ones on the boundary have only a single black vertex contained within $R$.  The error will be controlled in terms of the area of the dark quadrilaterals.}
	\label{RQFig}
\end{figure}

First we require an additional lemma.

\begin{lemma}\label{OtherAreaLem}
	Given a $K$-round quadrilateral $f=[z_1z_2z_3z_4]$ with orthogonal diagonals, let $z'$ be the intersection of the intersection of the lines orthogonally bisecting the diagonals of $f$ ($\Re(\overline{(z_1-z_3)}(z''-(z_1+z_3)/2)) = \Re(\overline{(z_2-z_4)}(z''-(z_2+z_4)/2)) = 0$) and $z''$ be the intersection point of the lines $\Re((z_1-z_3)(z''-(z_1+z_3)/2)) = \Re((z_2-z_4)(z''-(z_2+z_4)/2)) = 0$.  Then
	\[
	|\area(z_1z_2z'z_4)| \le C_K \area(z_1z_2z_3z_4)
	\]
	where the area is taken as the signed area, and
	\[
	\Im[(2z''-z_2-z_4)(z_4-z_2)] \le C_K \area(z_1z_2z_3z_4).
	\]
\end{lemma}
\begin{proof}
	In both cases of $z'$ and $z''$, the two lines used to define those points are orthogonal (since $z_1-z_3$ and $z_2-z_4$ are orthogonal).  In both cases, one line passes through $(z_2+z_4)/2$ and the other through $(z_1+z_3)/2$.  Thus, by the converse to Thales' theorem in Euclidean geometry, both the point $z'$ and $z''$ must lie on the circle whose diameter passes through the points $(z_2+z_4)/2$ and $(z_1+z_3)/2$, which is the circle of radius $|z_2+z_4-z_1-z_3|/4 \le \diam(f)/2$ around the point $(z_1+z_2+z_3+z_4)/4$. The other points are all contained in a ball of radius $3\diam(f)/4$ about the same point since
	\[
	|(z_1+z_2+z_3+z_4)/4 - z_1| \le |z_2-z_1|/4+|z_3-z_1|/4+|z_4-z_1|/4 \le 3\diam(f)/4.
	\]
	And hence $|\area(z_1z_2z'z_4)| \le C\diam(f)^2$ and $\Im[(2z''-z_2-z_4)(z_4-z_2)] \le C' \diam(f)^2$.
	
	By Lemma \ref{AreaLem}, we know that $\diam(f)^2 \le C_K \area(f)$ and thus the result follows.
\end{proof}

We may now prove the desired cases.

\begin{proof}[Proof of Proposition~\ref{LapProp} for $|z|^2$.]
	As discussed above, we wish to show that
	\[
	\sum_{w \in R \cap Q^\bullet} [\Delta_Q(|z|^2)](w)
	\]
	approximates $4\cdot\area(R)$.  Note that for $z'$ as defined in Lemma~\ref{OtherAreaLem} we have that given a face $f$, $\nabla_Q|z-z'|^2 = 0$ on $f$.  We will make frequent use of the identity 
	\[
	|z-a|^2 = |z-b|^2 + 2\Re[\overline{(b-a)}(z-b)] + |b-a|^2
	\]
	to shift our quadratic to be based around other points.
	
	First note that for any $w \in Q^\bullet$ we have that
	\[
		[\Delta_Q(|z|^2)](w) = [\Delta_Q(|z-w|^2 + 2\Re[\overline{w}(z-w)] + |w|^2)](w) = [\Delta_Q(|z-w|^2)](w) 
	\]
	since since the last two terms are discrete harmonic.  Now, by Lemma~\ref{LapExpressLem} , and expanding around the point $z'$ from Lemma~\ref{OtherAreaLem}
	\begin{align*}
		\hspace{3em} & \hspace{-3em} [\Delta_Q(|z-w|^2)](w) \\
		& = \sum_{\substack{f=[z_1z_2z_3z_4]\\z_1 = w}} * \nabla_Q(|z-w|^2)(f) \odot (z_4-z_2) \\
		& = \sum_{\substack{f=[z_1z_2z_3z_4]\\z_1 = w}} * \nabla_Q(|z-z'|^2 + 2\Re[\overline{(z'-w)}(z-z')] + |z'-w|^2)(f) \odot (z_4-z_2) \\
		& = \sum_{\substack{f=[z_1z_2z_3z_4]\\z_1 = w}} * \nabla_Q(2\Re[\overline{(z'-w)}(z-z')])(f) \odot (z_4-z_2) \\
		& = \sum_{\substack{f=[z_1z_2z_3z_4]\\z_1 = w}} 2 \; \Im[\overline{(z'-w)}(z_4-z_2)] \\
		& = \sum_{\substack{f=[z_1z_2z_3z_4]\\z_1 = w}} 4 \; \area(z_1z_2z'z_4)
	\end{align*}
	where the area is again interpreted as the signed algebraic area.
	
	We now sum over all $w \in R \cap Q^\bullet$.  Since $\area(z_1z_2z'z_4)+\area(z_3z_4z'z_2) = \area(z_1z_2z_3z_4)$ holds for any collection of points, we know that for any face with both black vertices contained in $R$, the sum over the discrete Laplacian will pick up a term for that face from both black vertices and that those two terms will sum to $4\,\area(z_1z_2z_3z_4)$.  Thus we have that
	\begin{align*}
		\sum_{w \in R \cap Q^\bullet} [\Delta_Q(|z|^2)](w) & = \sum_{f:f\cap Q^\bullet \subset R} 4 \; \area(f) + \sum_{\substack{f=[z_1z_2z_3z_4]\\z_1\in Q^\bullet\cap R\\ z_3 \not \in Q^\bullet\cap R}} 4\, \area(z_1z_2z'z_4) \\
		& = 4\,\area(R_Q) + \sum_{\substack{f=[z_1z_2z_3z_4]\\z_1\in Q^\bullet\cap R\\ z_3 \not \in Q^\bullet\cap R}} 4\,\area(z_1z_2z'z_4)
	\end{align*}
	where $R_Q$ is the union of all faces with both black vertices within $R$.
	The second term above may be bounded in absolute value by
	\[
	C_K\sum_{\substack{f=[z_1z_2z_3z_4]\\z_1\in Q^\bullet\cap R\\ z_3 \not \in Q^\bullet\cap R}} 4\,\area(z_1z_2z_3z_4)
	\]
	which is the area of a union of faces contained entirely within distance $M(Q)$ of $\bd R$ and thus may be bounded by $C_K M(Q)r$.
	
	The first term must be compared with $4\,\area(R)$.  In this case the symmetric difference $R \triangle R_Q$ is again contained entirely within distance $M(Q)$ of $\bd R$ and thus 
	\[
	|4\,\area(R) - 4\,\area(R_Q)| \le 4\,\area(R \triangle R_Q) \le C M(Q)r.
	\]  Combining these shows that
	\[
	\Big|\sum_{z \in R \cap Q^\bullet} [\Delta_Q(g|_{Q^*})](z) - \int_R\Delta g \; dx dy\Big| \le C_KM(Q)r
	\]
	as needed.
\end{proof}

We now modify this proof to obtain the proof for $\Re\; z^2$ (the proof for $\Im \;z^2$ is analogous).

\begin{proof}[Proof of Proposition~\ref{LapProp} for $\Re \;z^2$.]
	We wish to show that
	\[
	\sum_{w \in R \cap Q^\bullet} [\Delta_Q(\Re \; z^2)](w)
	\]
	is zero aside from terms contributed from the boundary.  Note that for $z''$ as defined in Lemma~\ref{OtherAreaLem} we have that given a face $f$, $\nabla_Q\Re(z-z'')^2 = 0$ on $f$.  We will make frequent use of the identity 
	\[
	\Re(z-a)^2 = \Re(z-b)^2 + 2\Re[(b-a)(z-b)] + \Re(b-a)^2
	\]
	to shift our quadratic to be based around other points.
	
	First note that for any $w \in Q^\bullet$ we have that
	\[
		[\Delta_Q(\Re \; z^2)](w) = [\Delta_Q(\Re(z-w)^2 + 2\Re[w(z-w)] + \Re \; w^2)](w) = [\Delta_Q(\Re(z-w)^2)](w) 
	\]
	since since the last two terms are discrete harmonic.  Now, by Lemma~\ref{LapExpressLem}, and expanding around the point $z''$ from Lemma~\ref{OtherAreaLem}
	\begin{align*}
		\hspace{2em} & \hspace{-2em} [\Delta_Q(\Re(z-w)^2)](w) \\
		& = \sum_{\substack{f=[z_1z_2z_3z_4]\\z_1 = w}} * \nabla_Q(\Re(z-w)^2)(f) \odot (z_4-z_2) \\
		& = \sum_{\substack{f=[z_1z_2z_3z_4]\\z_1 = w}} * \nabla_Q(\Re(z-z'')^2 + 2\Re[(z''-w)(z-z'')] + \Re(z''-w)^2)(f) \odot (z_4-z_2) \\
		& = \sum_{\substack{f=[z_1z_2z_3z_4]\\z_1 = w}} * \nabla_Q(2\Re[(z''-w)(z-z'')])(f) \odot (z_4-z_2) \\
		& = \sum_{\substack{f=[z_1z_2z_3z_4]\\z_1 = w}} 2 \; \Im[(z''-w)(z_4-z_2)] \\
		& = \sum_{\substack{f=[z_1z_2z_3z_4]\\z_1 = w}} 2 \; \Im[z''(z_4-z_2)]
	\end{align*}
	where the last equality holds since $w \in Q^\bullet \cap R$ implies that the faces $f$ in the sum form a full cycle around $w$.
	
	We now sum over all $w \in R \cap Q^\bullet$.  Note that the two terms corresponding to two black vertices of the same face are
	\[
	\Im[z''(z_4-z_2)] + \Im[z''(z_2-z_4)] =  0,
	\]
	and thus the sum becomes
	\[
	\sum_{w \in R \cap Q^\bullet} [\Delta_Q(\Re \; z^2)](w) = \sum_{\substack{f=[z_1z_2z_3z_4]\\z_1\in Q^\bullet\cap R\\ z_3 \not \in Q^\bullet\cap R}} 2 \; \Im[z''(z_4-z_2)].
	\]
	Since the edges $(z_4-z_2)$ in the above formula form a closed cycle, we may, by summation by parts, subtract the sum of $(z_2+z_4)(z_4-z_2)$ to obtain
	\begin{align*}
	\Big|\sum_{w \in R \cap Q^\bullet} [\Delta_Q(\Re \; z^2)](w)\Big| & = \Big|\sum_{\substack{f=[z_1z_2z_3z_4]\\z_1\in Q^\bullet\cap R\\ z_3 \not \in Q^\bullet\cap R}} 2 \; \Im[(2z''-z_2+z_4)(z_4-z_2)]\Big| \\
	& \le C_K \sum_{\substack{f=[z_1z_2z_3z_4]\\z_1\in Q^\bullet\cap R\\ z_3 \not \in Q^\bullet\cap R}} \area(z_1z_2z_3z_4) \\
	& \le C_K M(Q)r
	\end{align*}
	where the second to last line follows by Lemma~\ref{OtherAreaLem} and the last line follows since it is the area of a union of faces contained within distance $M(Q)$ of $\bd R$.
\end{proof}

Thus, all that remains is to show that if the function is $\tilde g(z)$ in the above expansion where $D^k\tilde g(0) = 0$ for $k = 0,1,2$.  This case is simpler than the two previous ones since the result will follow by showing that each of the terms on the left hand side of Proposition~\ref{LapProp} are themselves small enough.  To do so we must first estimate the difference between the discrete and continuous gradient.

\begin{lemma}\label{GradApproxLem}
	For any $K$-round quadrilateral $f=[z_1z_2z_3z_4]$ we have
	\[
	|\nabla g - \nabla_Q(g|_{Q^*})| \le C \diam(f) \max_{z \in \mathrm{conv}(f)} |D^2 g(z)|.
	\]
	Where $\mathrm{conv}(f)$ is the convex hull of $f$.
\end{lemma}
\begin{proof}
	By Rolle's theorem, there is a point $z$ on the interval between $z_1$ and $z_3$ and another point $z'$ on the interval between $z_2$ and $z_4$ such that
	\[
	(\nabla g(z) - \nabla_Q g (f))\odot \frac{z_3 - z_1}{|z_3-z_1|} = 0 \ \text{and} \ (\nabla g(z') - \nabla_Q g (f))\odot \frac{z_4 - z_2}{|z_4-z_2|} = 0
	\]
	Thus, by integration, we see that
	\[
	|(\nabla g(z_1) - \nabla_Q g (f))\odot \frac{z_3 - z_1}{|z_3-z_1|}| \le C \diam(f) \max_{z \in \mathrm{conv}(f)} |D^2 g(z)| 
	\]
	and
	\[ |(\nabla g(z_1) - \nabla_Q g (f))\odot \frac{z_4 - z_2}{|z_4-z_2|}| \le C \diam(f) \max_{z \in \mathrm{conv}(f)} |D^2 g(z)|.
	\]
	Thus
	\[
	|(\nabla g(z_1) - \nabla_Q g (f))| \le C \diam(f) \max_{z \in \mathrm{conv}(f)} |D^2 g(z)|
	\]
	as desired by orthogonality.
\end{proof}

\begin{proof}[Proof of Proposition~\ref{LapProp} for $\tilde g$.]
	First note that, given the conditions on $\tilde g$, the following three estimates hold by integration:
	\[
	|\nabla\tilde g(z)| \le Cr^2\max_{z \in R} |D^3\tilde g(z)|, \quad |D^2\tilde g(z)| \le C r \max_{z \in R} |D^3\tilde g(z)|, \quad |\Delta \tilde g(z)| \le  C r \max_{z \in R} |D^3\tilde g(z)|.
	\]
	By integrating the third inequality over $R$ we immediately see that
	\[
	\Big| \int_R \Delta \tilde g \; dxdy \Big| \le Cr^3\max_{z \in R} |D^3\tilde g(z)|.
	\]
	
	We now produce the same bound for the discrete case.  First note that
	\begin{align*}
	\Big|\sum_{w \in R \cap Q^\bullet} [\Delta_Q(\tilde g(z))](w)\Big| & = \Big| \sum_{w \in R \cap Q^\bullet} \sum_{\substack{f = [z_1z_2z_3z_4] \\ z_1 = w}} * \nabla_Q \tilde g(f) \odot (z_2-z_4)\Big| \\ 
	& = \Big|  \sum_{\substack{f = [z_1z_2z_3z_4] \\ z_1 \in R \cap Q^\bullet \\ z_3 \not \in R \cap Q^\bullet}} * \nabla_Q \tilde g(f) \odot (z_2-z_4)\Big| \\ 
	& \le \sum_{\substack{f = [z_1z_2z_3z_4] \\ z_1 \in R \cap Q^\bullet\\ z_3 \not \in R \cap Q^\bullet}} (|\nabla_Q \tilde g(f) - \nabla \tilde g (z_1)| + |\nabla \tilde g(z_1)|) \cdot |z_4-z_1|
	\end{align*}
	By using Lemma~\ref{GradApproxLem} followed by the estimate on $|D^2 \tilde g(z)|$, and $|\nabla \tilde g(z)|$ we see that
	\[
	\Big|\sum_{w \in R \cap Q^\bullet} [\Delta_Q(\tilde g(z))](w)\Big| \le (C_K \diam(f)r + r^2) \max_{z \in R} |D^3\tilde g(z)|  \sum_{\substack{f = [z_1z_2z_3z_4] \\ z_1 \in R \cap Q^\bullet\\ z_3 \not \in R \cap Q^\bullet}} |z_4-z_1|.
	\] 
	We obtain the final estimate by an application of Proposition~\ref{DiamProp}, to see that
	\[
	\Big|\sum_{w \in R \cap Q^\bullet} [\Delta_Q(\tilde g(z))](w)\Big| \le C_Kr^3 \max_{z \in R} |D^3\tilde g(z)|.
	\]
\end{proof}

\section{Convergence results}

With all of these estimates in place, we may now establish the desired convergence results.  Again, many of these proofs closely mirror those in \cite{Skopenkov}.  Before establishing convergence of the functions, we must first establish the convergence of the discrete energy to the continuous one.

\begin{lemma}\label{EnergyConvLem}
	Let $\Omega$ be a bounded simply-connected domain with smooth boundary.  Let $\{Q_n\}$ be a sequence of $K$-round quadrilateral lattices approximating $\Omega$.  Then for any $C^2(\C)$ smooth function $g: \C \rightarrow \R$, $E_{Q_n}(g|_{Q_n^*}) \rightarrow E_\Omega(g)$ as $n \rightarrow \infty$.
\end{lemma}

\begin{proof}
	Let $\hat Q_n$ denote the region obtained by taking the union of all faces of $Q_n$ and all interior edges.  By Definition~\ref{ApproxDef}, since the $Q_n$ approximate $\Omega$, we have that $\bd Q_n$ is contained within the $\epsilon_n$-neighborhood of $\bd\Omega$ and $\bd\Omega$ is contained within the $\epsilon_n$-neighborhood of $\bd Q_n$ for some $\epsilon_n \rightarrow 0$ as $n \rightarrow 0$.  This implies, when combined with the smoothness of $\bd \Omega$, that $\area(\Omega \setminus \hat Q_n),\area(\hat Q_n \setminus \Omega) \rightarrow 0$ as $n \rightarrow \infty$ as both differences are contained within the $\epsilon_n$-neighborhood of $\Omega$.  Moreover, there exists a compact convex neighborhood $\Omega'$ of $\Omega$ which contains all $\hat Q_n$.  Since $g$ is $C^2(\C)$, we know that $\nabla g$ is bounded on $\Omega'$ and thus the integrals defining $E_\Omega(g),E_{\hat Q_n}(g)$ exist and moreover $E_\Omega(g) - E_{\hat Q_n}(g) = E_{\Omega\setminus \hat Q_n}(g) - E_{\hat Q_n\setminus \Omega}(g) \rightarrow 0$ as $n \rightarrow 0$.  Lemma~\ref{GradApproxLem} now completes the proof.
\end{proof}

We now show that, if we are given a collection discrete harmonic functions which do converge, then the limit must itself be harmonic.

\begin{proposition}\label{HarmConvProp}
Let $\{Q_n\}$ be a sequence of $K$-round orthogonal quadrilateral lattices approximating a domain $\Omega$.  Let $u_n:Q_n^* \rightarrow \R$ be a sequence of discrete harmonic functions such that $u_n$ converges uniformly to a continuous $u:\Omega \rightarrow \R$.  Then the function $u$ is harmonic.
\end{proposition}
\begin{proof}
	By Weyl's lemma (see, for example, \cite[p. 235]{Stein}), it will suffice to show that
	\[
	\int_{\Omega} u \nabla v \; dx dy = 0
	\]
	for any $C^\infty$ function $v$ whose support $S$ is compactly contained in $\Omega$. 
	
	First, note by Green's identity (Lemma~\ref{GreenLem}) that, as long as $n$ is taken to be sufficiently large that $v|_{Q^*}$ vanishes on $\bd Q_n$ we have immediately that
	\[
	\sum_{z \in Q^\bullet_n} [u_n\nabla_{Q_n}v](z) = \sum_{z \in Q^\bullet_n} [v\nabla_{Q_n}u_n](z) = 0.
	\]
	This tells us that the discrete version of Weyl's lemma holds for the discrete harmonic functions.  We need only bound the error between the discrete sum and the integral to conclude that Weyl's lemma holds for the limit function.
	
	To do so, we will include an intermediate lattice $\sqrt{2M(Q_n)}\Z^2$ at a scale which is large compared to the lattice spacing on $Q_n$ but still small compared to $\Omega$ (the exact choice of $\sqrt{2M(Q_n)}$ is taken simply for convenience).  We define a function on the rectangular faces $R$ of this intermediate lattice by $\tilde u_n(R) = \max_{z \in R \cap S} u(z)$.  By continuity, we know that $\tilde u_n$, extended to a function on $\C$ by taking it constant on faces, converges uniformly to $u$ on the support of $v$.  Thus, we have, by the uniform converges of both $u_n$ and $\tilde u_n$ to $u$, that
	\[
	\Big|\int_\Omega u\Delta v \; dxdy - \sum_{z \in Q_n^\bullet} [u_n\Delta]\Big| \le \sum_{R \ : \ R \cap S \neq \emptyset} (|\tilde u_n(R)|+o(1)) \cdot \Big|\int_R \Delta v \; dxdy - \sum_{z \in R \cap Q^\bullet_n}[\Delta_{Q_n}v](z)\Big|.
	\]
	
	By Proposition~\ref{LapProp}, we see that 
	\begin{align*}
	\hspace{3em} & \hspace{-3em} \sum_{R \ : \ R \cap S \neq \emptyset} (|\tilde u_n(R)|+o(1)) \cdot \Big|\int_R \Delta v \; dxdy - \sum_{z \in R \cap Q^\bullet_n}[\Delta_{Q_n}v](z)\Big| \\
	& \le \sum_{R \ : \ R \cap S \neq \emptyset} (|\tilde u_n(R)|+o(1)) \cdot C_K M(Q_n)^{3/2} \Big(\max_{z \in R}|D^2v(z)|+\max_{z \in R}|D^3v(z)|\Big) \\
	& \le C_K \cdot \area(S) \cdot \Big(\max_{z \in S} u(z)\Big) \cdot \Big(\max_{z \in S}|D^2v(z)|+\max_{z \in S}|D^3v(z)|\Big) \cdot M(Q_n)^{1/2} \\
	& = O(M(Q_n)^{1/2}) \rightarrow 0,
	\end{align*}
	where the second to last line follows by estimating the number of squares of side length $M(Q)^{1/2}$ needed to cover $S$.
\end{proof}

We may now use these results to establish the final limit theorem.

\begin{theorem}\label{FullTheorem}
	Let $\{Q_n\}_{n \in \N}$ be a sequence of $K$-round quadrilateral lattices approximating a bounded simply connected domain $\Omega \in \C$ and $g:\C \rightarrow \R$ be a given smooth boundary value.  Then, the sequence of solutions $u_{n}$ to the discrete Dirichlet problem on $Q_n$ with boundary values $g |_{\bd Q_n}$ uniformly converges to the solution $u$ to the Dirichlet problem on $\Omega$ with boundary values $g|_{\bd \Omega}$.
\end{theorem}
\begin{proof}
	First note that because the $Q_n$ approximate the bounded domain $\Omega$, all the lattices $Q_n$ are contained in some large ball $V$.  Thus, by the maximum principle (Lemma~\ref{MaxLem}), we know that the $|u_n|$ are uniformly bounded by $\max_{w \in V} |g(v)| < \infty$.
	
	We now show that the family of functions $\{u_n\}_{n \in \N}$ are equicontinuous, which is to say that there exists some positive function $\delta(\epsilon)$ such that for every $n$ we have that for every $z,w \in Q_n^\bullet$, we have that $|z-w| < \delta(\epsilon)$ implies that $|u_n(z) - u_n(w)| < \epsilon$.   We do so by invoking Proposition~\ref{EquiProp}.  First, suppose we are in the case that $M(Q_n) < |z-w|$. When $R = (\diam(V)|z-w|)^{1/2}$ we have that
	\begin{align*}
		|u(z)-u(w)| & \le C_K E_{Q_R}(u)^{1/2}\log^{-1/2}\Big[\frac{R}{|z-w|}\Big] + \max_{z',w' \in \bd Q \cap B_R((z+w)/2)} |u(z')-u(w')| \\
		& \hspace{-6.1em} \le C_K (E_{V}(g)+o(1))^{1/2}\log^{-1/2}[\diam(V)^{1/2}|z-w|^{-1/2}] + (\diam(V)|z-w|)^{1/2}\cdot\max_{z'\in V}|D^1g(z')| \\
	\end{align*}
which is bounded uniformly in $n$ in terms of $|z-w|$ and tends to zero as $|z-w|$ tends to zero.  Note that we have used Lemma~\ref{EnergyConvLem} to conclude that the energy is bounded by its continuous counterpart.  If instead $|z-w| < M(Q_n)$, then by setting $R = (\diam(V)M(Q_n))^{1/2}$ we obtain the same bound with $M(Q_n)$ replacing $|z-w|$.  Note that selecting $|z-w|< \delta_0$ we may make $M(Q_n)$ as small as we wish since for any $\epsilon_0$ there are only finitely many $n$ for which $M(Q_n) > \epsilon_0$ and thus we may take $\delta_0 = \frac{1}{2}\min_{z,w,n \ : \ z,w \in Q_n^*} |z-w|$.  These two estimates prove the desired equicontinuity.

Now, by Arzel\`a-Ascoli, we know that there exists a subsequence of the $u_n|_{Q^\bullet_n}$ converges uniformly to a $u$ continuous on the closure of $\Omega$.  By Proposition~\ref{HarmConvProp}, we know that the limit function is harmonic in $\Omega$.  Moreover, by Definition~\ref{ApproxDef}, we know that for any $z \in \bd \Omega$, there exists a sequence of points $z_n \in \bd Q_n \cap Q_n^\bullet$ such that $z_n \rightarrow z$ as $n \rightarrow \infty$, and thus we know that $u|_{\bd \Omega} = g|_{\bd \Omega}$ and hence that $u$ solves the Dirichlet problem on $\Omega$ with boundary values $g|_{\bd \Omega}$.  Since this limit is unique, we know that had we passed to any arbitrary subsequence before applying Arzel\`a-Ascoli, we would have obtained the same limit, and thus we know that the entire sequence $u_n|_{Q^\bullet_n}$ converges uniformly to $u$ as desired.
\end{proof}

\bibliographystyle{plain}
\bibliography{Bib}

\begin{thebibliography}{10}

\bibitem{CRT}
D.~Aldous.
\newblock The continuum random tree. {III}.
\newblock {\em Ann. Probab.}, 21(1):248--289, 1993.

\bibitem{Andreev}
E.~Andreev.
\newblock Convex polyhedra in {L}obacevskii space.
\newblock {\em Math. USSR Sbornik}, 12:155--159, 1970.

\bibitem{Bernardi}
O.~Bernardi.
\newblock Bijective counting of tree-rooted maps and shuffles of parenthesis
  systems.
\newblock {\em Electron. J. Combin.}, 14(1):Research Paper 9, 36 pp.
  (electronic), 2007.

\bibitem{IsingSpin}
D.~Chelkak, C.~Hongler, and K.~Izyurov.
\newblock Conformal invariance of spin correlations in the planar {I}sing
  model.
\newblock {\em Ann. of Math. (2)}, 181(3):1087--1138, 2015.

\bibitem{ChelSmir}
D.~Chelkak and S.~Smirnov.
\newblock Discrete complex analysis on isoradial graphs.
\newblock {\em Adv. Math.}, 228(3):1590--1630, 2011.

\bibitem{ising}
D.~Chelkak and S.~Smirnov.
\newblock Universality in the 2{D} {I}sing model and conformal invariance of
  fermionic observables.
\newblock {\em Invent. Math.}, 189(3):515--580, 2012.

\bibitem{Cour}
R.~Courant, K.~Friedrichs, and H.~Lewy.
\newblock \"{U}ber die partiellen {D}ifferenzengleichungen der mathematischen
  {P}hysik.
\newblock {\em Math. Ann.}, 100(1):32--74, 1928.

\bibitem{Duffin}
R.~J. Duffin.
\newblock Potential theory on a rhombic lattice.
\newblock {\em J. Combin. Theory}, 5:258--272, 1968.

\bibitem{conflatt}
H.~Duminil-Copin and S.~Smirnov.
\newblock Conformal invariance of lattice models.
\newblock In {\em Probability and statistical physics in two and more
  dimensions}, volume~15 of {\em Clay Math. Proc.}, pages 213--276. Amer. Math.
  Soc., Providence, RI, 2012.

\bibitem{connective}
H.~Duminil-Copin and S.~Smirnov.
\newblock The connective constant of the honeycomb lattice equals
  {$\sqrt{2+\sqrt{2}}$}.
\newblock {\em Ann. of Math. (2)}, 175(3):1653--1665, 2012.

\bibitem{KPZ}
B.~Duplantier and S.~Sheffield.
\newblock Liouville quantum gravity and {KPZ}.
\newblock {\em Invent. Math.}, 185(2):333--393, 2011.

\bibitem{Ferrand}
J.~Ferrand.
\newblock Fonctions pr/'eharmoniques et fonctions pre/'eholomorphes.
\newblock {\em Bull. Sci. Math.}, 68:152--180, 1944.

\bibitem{Isaacs}
R.~Isaacs.
\newblock A finite difference function theory.
\newblock {\em Univ. Nac. Tucum\'an. Revista A.}, 2:177--201, 1941.

\bibitem{KKDraw}
T.~Kamada and S.~Kawai.
\newblock An algorithm for drawing general undirected graphs.
\newblock {\em Inf. Process. Lett.}, 31(1):7--15, April 1989.

\bibitem{kenyon1}
R.~Kenyon.
\newblock The asymptotic determinant of the discrete {L}aplacian.
\newblock {\em Acta Math.}, 185(2):239--286, 2000.

\bibitem{kenyon2}
R.~Kenyon.
\newblock Conformal invariance of domino tiling.
\newblock {\em Ann. Probab.}, 28(2):759--795, 2000.

\bibitem{kenyon3}
R.~Kenyon.
\newblock Conformal invariance of loops in the double-dimer model.
\newblock {\em Comm. Math. Phys.}, 326(2):477--497, 2014.

\bibitem{Kenyon4}
Ri. Kenyon and J.-M. Schlenker.
\newblock Rhombic embeddings of planar quad-graphs.
\newblock {\em Trans. Amer. Math. Soc.}, 357(9):3443--3458 (electronic), 2005.

\bibitem{Knuth}
D.~E. Knuth.
\newblock {\em The art of computer programming. {V}ol. 4A}.
\newblock Addison-Wesley, Upper Saddle River, NJ, 2011.
\newblock Combinatorial Algorithms, Part 1.

\bibitem{Koebe}
P.~Koebe.
\newblock Kontaktprobleme der konformen abbildung.
\newblock {\em Ber. S\"achs. Akad. Leipzig, Math.--Phys. Kl.}, 88:141--164,
  1936.

\bibitem{LERW}
Gregory~F. Lawler, Oded Schramm, and Wendelin Werner.
\newblock Conformal invariance of planar loop-erased random walks and uniform
  spanning trees.
\newblock {\em Ann. Probab.}, 32(1B):939--995, 2004.

\bibitem{LeGallRev}
J.-F. Le~Gall.
\newblock Random geometry on the sphere.
\newblock 2014.
\newblock arXiv: \href{http://arxiv.org/abs/1403.7943}{1403.7943v1}.

\bibitem{BuziosNotes}
J.-F. Le~Gall and G.~Miermont.
\newblock Scaling limits of random trees and planar maps.
\newblock In D.~Ellwood, C.~Newman, V.~Sidoravivius, and W.~Werner, editors,
  {\em Probability and Statistical Physics in Two and More Dimensions}.
  Providence, RI, Oxford, 2012.

\bibitem{Mercat1}
C.~Mercat.
\newblock Discrete {R}iemann surfaces and the {I}sing model.
\newblock {\em Comm. Math. Phys.}, 218(1):177--216, 2001.

\bibitem{QLE}
J.~Miller and S.~Sheffield.
\newblock Quantum {L}oewner {E}volution.
\newblock 2013.
\newblock arXiv: \href{http://arxiv.org/abs/1312.5745}{1312.5745v1}.

\bibitem{TreeMate2}
J.~Miller and S.~Sheffield.
\newblock An axiomatic characterization of the {B}rownian map.
\newblock 2015.
\newblock arXiv: \href{http://arxiv.org/abs/1506.03806v1}{1506.03806v1}.

\bibitem{TreeMate3}
J.~Miller and S.~Sheffield.
\newblock Liouville quantum gravity and the {B}rownian map i: The {QLE}(8/3,0)
  metric.
\newblock 2015.
\newblock arXiv: \href{http://arxiv.org/abs/1507.00719v1}{1507.00719v1}.

\bibitem{TreeMate1}
J.~Miller and S.~Sheffield.
\newblock Liouville quantum gravity spheres as matings of finite-diameter
  trees.
\newblock 2015.
\newblock arXiv: \href{http://arxiv.org/abs/1506.03804v1}{1506.03804v1}.

\bibitem{Mullin}
R.~C. Mullin.
\newblock On the enumeration of tree-rooted maps.
\newblock {\em Canad. J. Math.}, 19:174--183, 1967.

\bibitem{Remy}
J.-L. R\'emy.
\newblock Un proc\'ed\'e it\'eratif de d\'enombrement d\'arbres binaires et son
  application \`a leur g\'en\'eration al\'eatoire.
\newblock {\em RAIRO Inform. Th\'eor.}, 19(2):179--195, 1985.

\bibitem{RodinSullivan}
B.~Rodin and D.~Sullivan.
\newblock The convergence of circle packings to the {R}iemann mapping.
\newblock {\em J. Differential Geom.}, 26(2):349--360, 1987.

\bibitem{CWSheff}
S.~Sheffield.
\newblock Conformal weldings of random surfaces: {SLE} and the quantum gravity
  zipper.
\newblock 2010.
\newblock arXiv: \href{http://arxiv.org/abs/1012.4797}{1012.4797v1}.

\bibitem{HCSheff}
S.~Sheffield.
\newblock Quantum gravity and inventory accumulation.
\newblock 2011.
\newblock arXiv: \href{http://arxiv.org/abs/1108.2241}{1108.2241v1}.

\bibitem{Skopenkov}
M.~Skopenkov.
\newblock The boundary value problem for discrete analytic functions.
\newblock {\em Adv. Math.}, 240:61--87, 2013.

\bibitem{perc}
S.~Smirnov.
\newblock Critical percolation in the plane: conformal invariance, {C}ardy's
  formula, scaling limits.
\newblock {\em C. R. Acad. Sci. Paris S\'er. I Math.}, 333(3):239--244, 2001.

\bibitem{SmirDisc}
Stanislav Smirnov.
\newblock Discrete complex analysis and probability.
\newblock In {\em Proceedings of the {I}nternational {C}ongress of
  {M}athematicians. {V}olume {I}}, pages 595--621. Hindustan Book Agency, New
  Delhi, 2010.

\bibitem{Stein}
E.~M. Stein and R.~Shakarchi.
\newblock {\em Real analysis}.
\newblock Princeton Lectures in Analysis, III. Princeton University Press,
  Princeton, NJ, 2005.
\newblock Measure theory, integration, and Hilbert spaces.

\bibitem{CPSoft}
K.~Stephenson.
\newblock {C}ircle{P}ack [{C}omputer {S}oftware].
\newblock \url{http://www.math.utk.edu/~kens/CirclePack/}, 1992--2014.

\bibitem{CirclePack}
K.~Stephenson.
\newblock {\em Introduction to circle packing}.
\newblock Cambridge University Press, Cambridge, 2005.
\newblock The theory of discrete analytic functions.

\bibitem{Thurston}
W.~Thurston.
\newblock In {\em The geometry and topology of 3-manifolds}, chapter~13.
  Princeton University Press, Princeton, NJ, 1980.

\end{thebibliography}

\end{document}